\documentclass[10pt,a4paper,reqno]{amsart}

\usepackage{enumitem} 
\newlist{condenum}{enumerate}{1} 
\setlist[condenum]{label=(C\arabic*), 
                   ref=\arabic*, wide}
\usepackage{epsfig}
\usepackage{color}
\usepackage{amssymb,amsmath,amsthm,amstext,amsfonts}
\usepackage{amssymb,amsmath,amsthm,amstext,amsfonts}
\usepackage{amsmath,amstext,amsthm,amsfonts}
\usepackage{color}
\usepackage[mathscr]{eucal}
\usepackage{dsfont}
\usepackage{comment}
\usepackage{enumitem}

\usepackage{hyperref}
\RequirePackage[dvipsnames]{xcolor} 
\definecolor{halfgray}{gray}{0.55} 
\definecolor{webgreen}{rgb}{0,0.5,0}
\definecolor{webbrown}{rgb}{.6,0,0} \hypersetup{%
	colorlinks=true, linktocpage=true, pdfstartpage=3,
	pdfstartview=FitV,%
	breaklinks=true, pdfpagemode=UseNone, pageanchor=true,
	pdfpagemode=UseOutlines,%
	plainpages=false, bookmarksnumbered, bookmarksopen=true,
	bookmarksopenlevel=1,%
	hypertexnames=true,
	pdfhighlight=/O,
	urlcolor=webbrown, linkcolor=RoyalBlue,
	citecolor=webgreen, 
	pdftitle={},%
	pdfauthor={},%
	pdfsubject={2000 MAthematical Subject Classification: Primary:},%
	pdfkeywords={},%
	pdfcreator={pdfLaTeX},%
	pdfproducer={LaTeX with hyperref}%
}

\usepackage{psfrag}
\usepackage{url}
\usepackage{epstopdf}
\usepackage{epic,eepic}
\usepackage{upgreek}
\usepackage{amssymb}
\usepackage{mathrsfs}
\usepackage{verbatim}
\usepackage{mathrsfs}
\usepackage{amstext}
\usepackage{amsthm}
\usepackage{amssymb}
\usepackage{graphicx}

\usepackage[]{mdframed}

\usepackage{hyperref}
\RequirePackage[dvipsnames]{xcolor} 
\definecolor{halfgray}{gray}{0.55} 
\definecolor{webgreen}{rgb}{0,0.5,0}
\definecolor{webbrown}{rgb}{.6,0,0} \hypersetup{%
	colorlinks=true, linktocpage=true, pdfstartpage=3,
	pdfstartview=FitV,%
	breaklinks=true, pdfpagemode=UseNone, pageanchor=true,
	pdfpagemode=UseOutlines,%
	plainpages=false, bookmarksnumbered, bookmarksopen=true,
	bookmarksopenlevel=1,%
	hypertexnames=true,
	pdfhighlight=/O,
	urlcolor=webbrown, linkcolor=RoyalBlue,
	citecolor=webgreen, 
	pdftitle={},%
	pdfauthor={},%
	pdfsubject={2000 MAthematical Subject Classification: Primary:},%
	pdfkeywords={},%
	pdfcreator={pdfLaTeX},%
	pdfproducer={LaTeX with hyperref}%
}

  [1995/08/01 v0.1 Double stroke roman fonts]

\def\ds@whichfont{dsrom}
\relax

\DeclareMathAlphabet{\mathds}{U}{\ds@whichfont}{m}{n}

\newtheorem{theorem}{Theorem}[section]
\newtheorem{lemma}[theorem]{Lemma}
\newtheorem{corollary}[theorem]{Corollary}

\newtheorem{proposition}[theorem]{Proposition}
\theoremstyle{definition}
\newtheorem{definition}[theorem]{Definition}

\newtheorem{remark}[theorem]{Remark}

\theoremstyle{plain}

\theoremstyle{plain}
\theoremstyle{plain}
\theoremstyle{remark}
\newtheorem*{acknowledgement*}{Acknowledgement}

\DeclareMathOperator{\esssup}{ess\; sup}
\DeclareMathOperator{\essinf}{ess\; inf}

\newcommand{\cA}{{\mathcal A}}
\newcommand{\cB}{{\mathcal B}}

\newcommand{\cF}{{\mathcal F}}
\newcommand{\cG}{{\mathcal G}}

\newcommand{\cL}{{\mathcal L}}
\newcommand{\cM}{{\mathcal M}}

\newcommand{\cR}{{\mathcal R}}
\newcommand{\cS}{{\mathcal S}}
\newcommand{\cT}{{\mathcal T}}

\newcommand{\cX}{{\mathcal X}}

\newcommand{\te}{{\theta}}

\newcommand{\Om}{{\Omega}}
\newcommand{\om}{{\omega}}
\newcommand{\ve}{{\varepsilon}}
\newcommand{\del}{{\delta}}

\newcommand{\sig}{{\sigma}}
\newcommand{\al}{{\alpha}}

\newcommand{\la}{{\lambda}}

\newcommand{\bbC}{{\mathbb C}}

\newcommand{\bbL}{{\mathbb L}}
\newcommand{\bbN}{{\mathbb N}}
\newcommand{\bbP}{{\mathbb P}}
\newcommand{\bbR}{{\mathbb R}}

\newcommand{\bbZ}{{\mathbb Z}}

\newcommand{\var}{\text{var}}

\usepackage{orcidlink}

\begin{document}
\title[Liv\v sic regularity]{Liv\v sic regularity for random and sequential dynamics through transfer operators}

\author[Lucas Backes]{Lucas Backes \orcidlink{0000-0003-3275-1311} }
\address{\noindent Departamento de Matem\'atica, Universidade Federal do Rio Grande do Sul, Av. Bento Gon\c{c}alves 9500, CEP 91509-900, Porto Alegre, RS, Brazil.}
\email[Lucas Backes]{lucas.backes@ufrgs.br} 

\author[Davor Dragi\v cevi\' c]{Davor Dragi\v cevi\' c \orcidlink{0000-0002-1979-4344}} 
\address{Faculty of Mathematics, University of Rijeka, Croatia}
\email[Davor Dragi\v cevi\' c]{ddragicevic@math.uniri.hr}

\author[Yeor Hafouta]{Yeor Hafouta}
\address{Department of Mathematics, University of Florida, USA}
\email[Yeor Hafouta]{yeor.hafuta@ufl.edu}


\maketitle

\begin{abstract}
We prove Liv\v sic-type regularity results of coboundary representations for non-autonomous dynamical systems. Our results have an abstract nature and apply to several important specific situations, such as (higher-dimensional) random or sequential piecewise expanding maps and subshifts of finite type, which have applications to Markov interval maps and to finite state inhomogeneous elliptic Markov shifts, via symbolic representations. We also obtain results for some classes of non-autonomous hyperbolic systems. Our results can be seen as non-autonomous versions of a recent result obtained by Morris. However, we emphasize that our proof differs from the one mentioned previously even in the deterministic case. Finally, we show that our results provide a more relaxed characterization for having variance growth of Birkhoff sums on random and sequential dynamical systems; we show that such growth can fail only when the underlying functions are a coboundary without special restrictions on the regularity of the coboundary. For random systems, we show that this is equivalent to having a coboundary with bounded ``variation", but for sequential systems it turns out that this is no longer true, as demonstrated by examples.
\end{abstract}

\section{Introduction}

Given a measurable transformation $T\colon X\to  X$ acting on a probability space $(X, \mathcal{G},m)$, two measurable maps $F,B \colon X \to G$, where $(G,\cdot)$ is a topological group, are said to be a \emph{cohomologous} if there exists a measurable map $H\colon X \to G$, usually called \emph{transfer map}, such that
\begin{equation}\label{eq: cohom intro}
F(x)=H(Tx)B(x)H(x)^{-1}
\end{equation}
for $m$-a.e. $x\in X$. In the particular case when $B$ is equal to $e_G$, the identity element of $G$, we say that $F$ is a \emph{coboundary}. Equation \eqref{eq: cohom intro} is known as a \emph{cohomological equation}. 

The study of cohomological equations arises naturally in many areas of dynamical systems and has seen applications in a variety of problems: quasi-periodic dynamics, smoothness of invariant measures and conjugacies, mixing properties of suspended flows, time change for flows, rigidity of group actions, growth of (the variance of)  Birkhoff sums and many others where one reduces the original problem to cohomological considerations (see \cite{KN11}). In this direction, two of the main problems are determining whether certain maps are cohomologous and studying the regularity properties of the transfer map $H$. In the present paper, we are interested in the latter in the particular case when the group $G$ is $(\mathbb{R},+)$ and $B\equiv 0$ \footnote{Observe that when $G$ is commutative, the study of \eqref{eq: cohom intro} can be reduced to the study of an equation of the form $\tilde F(x)=H(Tx)H(x)^{-1}$ simply by considering $\tilde F(x)=F(x)B(x)^{-1}$. In particular, two cocycles are cohomologous if and only if their difference is a coboundary.}.
Specifically, we are interested in finding sufficient conditions such that if $F\colon X\to \mathbb R$ has a particular regularity and a measurable map $H\colon X\to \mathbb R$ satisfies
\begin{equation}\label{eq: cob real}
F(x)=H(Tx)-H(x)
\end{equation}
for $m$-almost every $x\in X$, then $H$ must also have improved regularity.

We observe that this problem has already been studied in a variety of settings, especially when the base dynamics $T$ exhibits extra properties such as some form of hyperbolicity. For example, when $T$ is a hyperbolic map and $F$ is H\"older continuous, any measurable solution $H$ has a version that is also H\"older continuous, as shown in the seminal paper \cite{Liv72} by Liv\v sic. Similarly, if $T$ is a $C^r$ hyperbolic diffeomorphism and $F$ is a $C^r$ map for any non-integer $r>1$, every continuous solution of \eqref{eq: cob real} is $C^r$ (see, for instance, \cite{dlLMM86} and references therein). In the case where the base dynamics is partially hyperbolic, an important result is due to Wilkinson \cite{Wil13} and says that if $F$ is $C^k$ for $k\geq 2$ and $T$ is a partially hyperbolic, accessible and strongly $r$-bunched diffeomorphism for some $r < k-1$ or $r=1$ then any continuous solution of \eqref{eq: cob real} is $C^r$. However, for non-uniformly hyperbolic systems $(T,m)$, it is not always possible to get good regularity for $H$ on the entire space. For example, in \cite{Pol05}, Pollicott presented a Lipschitz map $F\colon X\to \mathbb{R}$ admitting a measurable solution of \eqref{eq: cob real} that does not have a H\"older continuous version. Nevertheless, he was able to show that, for $m=\text{Lebesgue}$, any measurable solution of \eqref{eq: cob real} is H\"{o}lder continuous on sets of arbitrary large measure. Similar results, but for more regular systems, were obtained by de la Llave \cite{dlL92}.

Finally, let us mention a recent result of Morris \cite{Morris} which is the most relevant to our work, where he describes conditions on the transfer operator $\cL$ associated with $T$ which implies an enhanced regularity of $H$. The assumptions in~\cite{Morris} require that $\cL$ admits a spectral gap on a suitable Banach space consisting of integrable functions. These conditions are quite general, and they are able to capture wide classes of mainly expanding maps. For additional Liv\v sic-type regularity results obtained using transfer operator techniques, we refer to \cite{Jenk,NP,PP,PY}.

In the present work, we obtain results in the spirit of \cite{Morris} in the context of random and sequential dynamical systems. More precisely, given an invertible ergodic measure-preserving map $\sigma \colon \Omega\to \Omega$ acting on a probability space $(\Omega, \mathcal{F}, \mathbb{P})$, and maps $T_\omega\colon X\to X$, $\omega\in \Omega$, we present sufficient conditions on the transfer operators associated to $(T_\omega)_{\omega\in \Omega}$ such that if $F\colon \Omega\times X\to \mathbb{R}$ has a particular regularity and $H\colon  \Omega\times X\to \mathbb{R}$ is a measurable map satisfying 
\begin{equation*}
F(\omega,x)=H(\sigma \omega, T_\omega x)-H(\omega,x)=H\circ\tau(\omega,x)-H(\om,x),
\end{equation*}
where $\tau(\omega,x)=(\sigma\omega, T_\omega x)$,
then $H$ has the same regularity as $F$ (see Theorem \ref{MT}). In fact, we provide a closed formula\footnote{ that when $\tau$ is ergodic, up to an additive constant there is always at most one measurable solution to $F=H\circ\tau-H$. In fact, if $F=G\circ\tau-G$ then $H-G$ is $\tau$-invariant and consequently constant.} for $H(\om,\cdot)$ by means of the iterates of the random transfer operators.
We also obtain an analogous result for sequential dynamical systems (see Theorem \ref{MT2}). Additionally, we exploit symbolic representations and obtain similar results for random and sequential small perturbations of a single hyperbolic map (see Theorems \ref{ThmSDS} and \ref{ThmRDS}), generalizing Liv\v sic's original result to non-autonomous hyperbolic systems.

Note that in the sequential case Liv\v sic theory has a somewhat different form. The reason for that is that Birkhoff sums can converge almost everywhere in the non-autonomous setting. In Theorem \ref{MT2} we show that there could be (up to centralization) only one sequential coboundary representation
\begin{equation}\label{SeqCob}
 F_j=H_{j+1}\circ T_j-H_j   
\end{equation}
of a sequence of regular functions $F_j$, where $(T_j)$ is the given sequence of transformations, and $H_j$ is a sequence of measurable functions. When \eqref{SeqCob} holds, we will provide a closed formula for $H_j$ that involves iterations of the sequential transfer operators. We refer the reader to Remark \ref{Rem0} for a discussion of the emergence of \eqref{SeqCob}. It will also follow that the coboundary part $H_j$ has a finite $L^p$ norm, for all finite $p\geq1$, and that, roughly speaking, $H_n-U_n-q_n\to 0$ for some sequence of functions $U_n$ with the same regularity as $F_n$, where $q_n$ is a sequence of centralizing numbers.  
However, in Remark \ref{Rem} we will provide examples showing that the coboundary $H_j$ might not have the same level of regularity as the given sequence of functions $F_j$. In Theorem \ref{MT2} we will also show that the given sequence $F_j$ admits a sequential martingale coboundary representation and that the martingale part converges. For random or deterministic dynamical systems, a converging martingale must vanish (see Lemma \ref{MartLem}), which is no longer the case in sequential setup. This is the reason why the function $H_j$ might not inherit the same type of regularity as the given sequence of functions $F_j$.

We emphasize that despite the somewhat technical statement of our main results, it applies to several important contexts (discussed in~\cite{DolgHaf} and~\cite{DGGV}), as explained throughout the text. Our arguments rely on the spectral perturbation theory for cocycles of the so-called twisted transfer operators,  which was developed (generalizing earlier works in the deterministic setting~\cite{GH,Nag1,Nag2}) in~\cite{DolgHaf} and~\cite{DGGV, DH, DS} for piecewise expanding sequential and random dynamics, respectively, and utilizes a martingale coboundary decomposition of a fiberwise centered version of $F$ described in \cite{DolgHaf,DH1}. In particular, even in the deterministic case, our proof differs from the one presented in \cite{Morris}. 
 
As an immediate consequence of our results we provide certain characterizations for variance growth of random and sequential Birkhoff sums (Corollaries \ref{Cor1} and \ref{Cor2}), which is more relaxed compared with the classical theory in the (weakly) stationary case, which states that the variance is bounded iff the underlying functions admit and $L^2$ coboundary decomposition (upon centering, see \cite{Kifer1998}).

\subsection{Outline of the proof}
Our proof in both the sequential and random cases begins like the proof in \cite{Morris}. We show that certain (random or sequential) linear functionals $\ell_t$ are equivariant with respect to the complex perturbation of the underlying transfer operator with parameter $t\in\bbR$ corresponding to the coboundary function. From that point our proof takes a different route (as Morris's arguments relied on spectral theory, and we do not have a single transfer operator). We use the latter  equivariance to show that the Birkhoff sums generated by the functions $F$ exhibit no variance growth. Using existing tools, this yields that we get a coboundary with a certain amount of regularity (depending on the case, sequential or random). Then we consider the difference  $\hat H$ between the original and the new coboundary part. Using a similar family of functionals $\hat\ell_t$, we are able to show that the characteristic function of $\hat H$  minus a converging sum (which vanishes in the random case) has the modulus $1$ at every point, and thus $\hat H$ minus the sum is constant.
 This is the reason that our results work without a priori conditions on the coboundary part; the characteristic function is always well defined.  The latter gives us explicit formulas for the original coboundary parts, from which our results follow. In Remarks \ref{Rem0} and \ref{Rem} we address the question of regularity in the sequential case, and provide examples that demonstrate the peculiarity of our results in the sequential setting. 
 
\section{Random dynamics}\label{RDS}

In this section, we present our first main result in the random setting. We start by introducing the relevant setting and notation.

\subsection{Setting}\label{sec: setting random}
Throughout this section $(\Omega, \mathcal F, \mathbb P)$ will be an arbitrary probability space and 
$\sigma \colon \Omega \to \Omega$ is an invertible ergodic measure-preserving transformation on $(\Omega, \mathcal F, \mathbb P)$. 
Furthermore, let $X$ be a compact metric space equipped with a Borel $\sigma$-algebra $\mathcal G$ and $\mathcal X\subset \Omega \times X$ an $\mathcal F\otimes \mathcal G$-measurable set  such that its fibers $X_\om:=\{x \in  X: \ (\om, x)\in  \mathcal X\}$ are compact sets. For $\om \in \Om$, the Borel $\sigma$-algebra on $X_\om$ will be denoted by $\mathcal G_\om$. Finally, we assume that for each $\om \in \Om$, there is a probability measure $m_\om$ on $(X_\om, \mathcal G_\om)$ such that the map $\om\to m_\om$ is measurable in the sense that $\om\to\int_{X_\om} g(\om,x)dm_\om(x)$ is measurable for every measurable function $g:\mathcal X \to\bbR$.
\begin{remark}
We only require the compactness of $X_\om$  in order to get the measurability of $\om\to X_\om$ with respect to the topology induced by the Hausdorff metric. However, any other form of measurability would work. In particular, we can also consider the case where $X_\om=X$ do not depend on $\omega$, and in this case we do not need any special assumptions on $X$.
\end{remark}

Let $T_\om \colon X_\om \to X_{\sigma \om}$, $\om \in \Om$ be a collection of maps satisfying the following two properties:
\begin{itemize}
\item the map $(\om, x)\to T_\om (x)$ is measurable with respect to the restriction of the $\sigma$-algebra $\mathcal F\otimes \mathcal G$ on $\mathcal X$;
\item for $\om \in \Om$, $T_\om^*m_\om$ is absolutely continuous with respect to $m_{\sigma \om}$, where $T_\om^*m_\om$ denotes the push-forward of $m_\om$ with respect to $T_\om$.
\end{itemize}
For $\om \in \Om$, let $\cL_\om \colon L^1(X_\om, m_\om)\to L^1(X_{\sigma \om}, m_{\sigma \om})$ denote the transfer operator associated with $T_\om$. It is characterized by the following duality relation:
\begin{equation}\label{dual}
\int_{X_{\sigma \om} }(\cL_\om \varphi)\psi\, dm_\om=\int_{X_\om}\varphi  (\psi \circ T_\om)\, dm_{\sigma \om}, 
\end{equation}
for $\varphi \in L^1(X_\om, m_\om)$ and $\psi \in L^\infty(X_{\sigma \om}, m_{\sigma \om})$. We assume here that the transfer operators are measurable in the sense that given a measurable function $g:\cX\to\bbR$ such that $g(\om,\cdot)\in L^1(m_\om)$ ($\bbP$-a.e.) the map $(\om,x)\to\mathcal L_\om(g(\om,\cdot))(x)$ from $\cX$ to $\bbR$ is measurable.
For $\om \in \Om$ and $n\in \bbN$, set 
\[
\cL_\omega^{(n)}:=\cL_{\sigma^{n-1}\omega}\circ \ldots \circ \cL_{\sigma \omega}\circ \cL_\omega,
\]
which is the transfer operator associated with
\[
T_\om^{(n)}:=T_{\sigma^{n-1}\om }\circ \ldots \circ T_{\sigma \om}\circ T_\om.
\]
We will also consider the associated skew-product transformation $\tau \colon \mathcal X\to \mathcal X$ given by 
\begin{equation}\label{tau}
\tau(\om, x)=(\sigma \om, T_\om(x)), \quad (\om, x)\in \mathcal X.
\end{equation}

In the sequel, we will assume that each fiber $X_\om$ is endowed with the notion of variation $\var_\om\colon L^1(X_\om, m_\om)\to [0, +\infty]$ that satisfies conditions $(V1)-(V9)$ of~\cite[p.1130]{DGGV} (taking $X=X_\om$ and $m=m_\om$).  We also assume that $C_\var>0$ in~\cite[(V3)]{DGGV} is independent of $\om$ and that $\var_\om(1)=0$ for $\omega \in \Omega$, where $1$ denotes the constant function on $X_\omega$ that takes only value $1$.
For $\om \in \Om$, we define 
\[
\cB_\om:=\{\varphi \in L^1(X_\om, m_\om): \ \var_\om(\varphi)<+\infty\}.
\]
Then $\cB_\om$ is a Banach space with respect to the norm
\[
\|\varphi\|_{\cB_\om}:=\|\varphi\|_{L^1(m_\om)}+\var_\om(\varphi).
\]
We require that for each $\omega \in \Omega$, $\cL_\omega$ is a bounded operator from $\cB_\omega$ to $\cB_{\sigma \omega}$. 

Moreover, suppose that $\tau$ admits a unique ergodic invariant probability measure $\mu$ on $\mathcal X$ which is absolutely continuous with respect to $m$ such that for $\mathbb P$-a.e. $\omega \in \Omega$, $d\mu_\omega=v_\omega\, dm_\omega$, $v_\omega\in \cB_\omega$ is density with respect to $m_\omega$ and \begin{equation}\label{essinf}\essinf v_\omega \ge c,\end{equation} where $c>0$ is independent of $\omega$. Here, $m$ is a probability measure on $\mathcal X$ given by 
\[
m(A)=\int_\Om m_\om (A_\om) \, d\mathbb P(\om)\quad \text{for $A\subset \mathcal X$ measurable,}
\]
where $A_\om:=\{x\in X_\om: (\om, x)\in A\}$. Moreover, the collection $\{\mu_\omega\}_{\omega \in \Omega}$ is such that 
\[
\mu(A)=\int_\Om \mu_\om (A_\om) \, d\mathbb P(\om)\quad \text{for $A\subset \mathcal X$ measurable.}
\]

Let $F\colon \mathcal X\to \mathbb R$ be a measurable map that satisfies the following properties:
\begin{itemize}
\item \begin{equation}\label{obs1}
F(\omega, \cdot)\in \cB_\om \quad \text{for $\omega \in \Omega$;}
\end{equation}
\item there exists $K>0$ such that  \begin{equation}\label{obs2}
\|F(\omega, \cdot)\|_{\cB_\om} \le K, \quad \text{for $\mathbb P$-a.e. $\omega \in \Omega$.}
\end{equation}
\end{itemize}
For $\theta \in \mathbb C$ and $\omega \in \Omega$, let $\cL_\omega^\theta \colon \cB_\om \to \cB_{\sigma \om}$ be a linear operator 
defined by 
\begin{equation}\label{tto}
\cL_\omega^{\theta} \varphi:=\cL_\omega(e^{\theta F(\omega, \cdot)}\varphi), \quad \varphi \in \cB_\om.
\end{equation}
Arguing as in the proof of~\cite[Lemma 3.2.]{DGGV}, we find that $\cL_\omega^\theta$ is a well-defined and bounded linear operator for  $\omega \in \Omega$ and $\theta \in \mathbb C$. By $\cB_\omega^*$ we denote the dual space of $\cB_\omega$ and $(\cL_\om^\theta)^* \colon \cB_{\sigma \omega}^*\to \cB_\omega^*$ will denote the dual operator of $\cL_\om^\theta$.

\begin{definition}\label{RPF}
We say that the pair $(\mathcal L, F)$ admits a \emph{random RPF triple} if there exists a neighborhood $U$ of $0$ in $\mathbb C$, and for each $\theta \in U$ there is a triplet $(\lambda_\omega^\theta, v_\omega^\theta, \phi_\omega^\theta)\in \mathbb C\times \cB_\om \times \cB_\om^*$, $\om \in \Om$  such that the following holds:
\begin{enumerate}
\item for $\mathbb P$-a.e. $\omega \in \Omega$ and every $\theta \in U$, 
\begin{equation}\label{triplet}
\cL_\omega^\theta v_\omega^\theta=\lambda_\omega^\theta v_{\sigma \omega}^\theta, \quad (\cL_\omega^\theta)^*\phi_{\sigma \omega}^\theta=\lambda_\om^\theta \phi_\omega^\theta \quad \text{and} \quad \phi_\om^\theta(v_\omega^\theta)=1;
\end{equation}
\item for $\mathbb P$-a.e. $\om \in \Om$, $\lambda_\om^0=1$, $v_\om^0=v_\om$ and $\phi_\om^0=m_\om$, where $m_\om$ is identified with the functional in $\cB_\om^*$ given by $\cB_\om\ni\varphi \mapsto \int_{X_\om} \varphi \, dm_\om$;
\item for $\mathbb P$-a.e. $\om \in \Om$, the maps $\theta \mapsto \lambda_\om^\theta$, $\theta \mapsto v_\om^\theta$ and $\theta \mapsto \phi_\om^\theta$ are analytic;
\item there is $C>0$ such that for $\mathbb P$-a.e. $\om \in \Om$ and all $\theta \in U$,
\begin{equation}\label{930}
\max \{|\lambda_\om^\theta|, \|v_\om^\theta\|_{\cB_\om}, \|\phi_\om^\theta\|_{\cB_\om^*}\}\le C;
\end{equation}
\item there are $c>0$ and $r\in (0, 1)$ such that for $\mathbb P$-a.e. $\omega \in \Omega$, $n\in \mathbb N$, $\varphi\in \cB_\om$ and $\theta \in U$,
\begin{equation}\label{newadd}
\left \| \mathcal L_\om^{\theta, (n)}\varphi-\left (\prod_{i=0}^{n-1}\lambda_{\sigma^i \om}^\theta \right )\phi_\om^{\theta}(\varphi)v_{\sigma^n \om}^\theta \right \|_{\cB_{\sigma^n \om}} \le cr^n\|\varphi\|_{\cB_\om};
\end{equation}
\item there exist $c, \tilde \delta>0$ such that for $\mathbb P$-a.e. $\omega \in \Omega$, $t\in [-\tilde \delta, \tilde \delta]$ and $n\ge n_0=n_0(\omega)$, 
\begin{equation}
\left |\prod_{k=0}^{n-1} \lambda_{\sigma^k \om}^{it}\right |\le e^{-cnt^2\Sigma^2},
\end{equation}
provided that $\Sigma^2>0$ where 
\begin{equation}\label{variance}
\Sigma^2:=\int_{\mathcal X}\tilde F^2\, d\mu+2\sum_{n=1}^\infty\int_{\mathcal X}\tilde F \cdot (\tilde F\circ \tau^n)\, d\mu
\end{equation}
and
\begin{equation}\label{tildeF}
\tilde F(\omega, \cdot):=F(\omega, \cdot)-\int_{X_\omega}F(\omega, \cdot)\, d\mu_\omega.
\end{equation}
\end{enumerate}

\end{definition}

\begin{remark}
\begin{itemize}
\item The acronym `RPF' in the preceding definition stands for `Ruelle-Perron-Frobenius'.
\item  We refer to \cite[Section 2.3.1]{DGGV} and \cite[Chapter 5]{HK} for examples of  cocycles $\mathcal L=(\cL_\omega)_{\omega \in \Omega}$ such that pairs $(F, \cL)$ admit a random RPF triple for any measurable map $F\colon \mathcal X\to \mathbb R$ satisfying~\eqref{obs1} and~\eqref{obs2}.
\item We note that the convergence of the series in~\eqref{variance} follows from the other assumptions in Definition~\ref{RPF} (applied for $\theta=0$) and~\eqref{obs2} as can be argued as in the proof of~\cite[Lemma 12]{DGGV1}. Moreover, in general, $\Sigma^2\ge 0$ and 
\begin{equation}\label{asymvar}
\lim_{n\to \infty}\frac 1 n\int_{X_\om} S_n \tilde F(\omega, \cdot)^2\, d\mu_\omega=\Sigma^2 \quad \text{for $\mathbb P$-a.e. $\omega \in \Omega$},
\end{equation}
where 
\begin{equation}\label{birksums}
S_n \tilde F(\omega, x):=\sum_{i=0}^{n-1}\tilde F(\sigma^i \omega, T_\omega^{(i)}(x)), \quad (\omega, x)\in\mathcal X.
\end{equation}
\item Observe that $\tilde F$ also satisfies~\eqref{obs1} and~\eqref{obs2} (with a different $K$). In addition, 
\[
\int_{X_\omega}\tilde F(\omega, \cdot)\, d\mu_\omega=0, \quad \text{for $\mathbb P$-a.e. $\omega \in \Omega$.}
\]
\item In the context of \cite[Section 2.3.1]{DGGV} and \cite[Chapter 5]{HK}, the triplets $(\lambda_\om^\theta, v_\omega^\theta, \phi_\omega^\theta)$ exhibit measurability with respect to $\omega$. To be more precise, in the examples discussed in~\cite[Section 2.3.1]{DGGV}, $\cB_\om=\cB$
 does not depend on $\omega$, and the maps $\omega \mapsto \lambda_\om^\theta \in \mathbb C$, $\omega \mapsto v_\omega^\theta\in \cB$ and $\omega \mapsto \phi_\omega^\theta \in \cB^*$ are measurable for each $\theta \in U$.
 \end{itemize}
\end{remark}

\subsection{Statement of the first main result}

We are now in a position to state the first main result of our paper.
\begin{theorem}\label{MT}
Suppose that a pair $(F, \mathcal L)$ admits a random RPF triplet. 
Furthermore, suppose that there exists a measurable map $H\colon\cX\to\mathbb R$ such that
\begin{equation}\label{FH}
F=H\circ \tau-H,
\end{equation}
where $\tau$ is given by~\eqref{tau}. Then, $H(\om, \cdot)\in \cB_\om$ for $\mathbb P$-a.e. $\om \in \Om$. Moreover,  there exists a constant $C>0$ such that $\var_\om(H(\omega, \cdot))\leq C$ for $\mathbb P$-a.e. $\omega \in \Omega$ and, in fact,
\begin{equation}\label{sumH}
H(\om,\cdot)=\int_{X_\om} H(\om, \cdot )\, d\mu_\om+\frac{1}{v_\om}\sum_{n=0}^\infty \mathcal L_{\sigma^{-n}\omega}^{(n)}(v_{\sigma^{-n}\omega}\tilde F(\sigma^{-n}\omega, \cdot)),\,\,\mathbb P-\text{a.e.,}
\end{equation}
where $\tilde F$ is given by~\eqref{tildeF}.
\end{theorem}
The following result is a standard application of Theorem \ref{MT}.
\begin{corollary}\label{Cor1}
In the circumstances of Theorem \ref{MT} we have $\Sigma^2=0$ if and only if 
\begin{equation}\label{Cobb}
 \tilde F=H\circ\tau-H,   
\end{equation}
for some measurable function $H$, where $\tilde F$ is given by~\eqref{tildeF}. In that case, we must have $H(\omega, \cdot)\in \cB_\omega$ and $\|H(\omega, \cdot)\|_{\mathcal B_\omega}\leq C$ for some constant $C>0$ and for $\mathbb P$-a.e. $\omega \in \Omega$. The same conclusion holds without centering $F(\omega,\cdot)$ (that is, with $F$ on the left-hand side of~\eqref{Cobb}), but with $\text{var}_\omega(H(\omega, \cdot))\leq C$  instead of $\|H(\omega, \cdot)\|_{\mathcal B_\omega}\leq C$. 
\end{corollary}
\begin{remark}
The classical characterization of $\Sigma^2=0$ that comes from the  theory of stochastic processes is that \eqref{Cobb} holds with $H\in L^2(\mu)$, see \cite{Kifer1998}. Here we show that $H$ can be replaced with an arbitrary measurable function, as well as with functions with bounded variation. 
\end{remark}

\subsection{Proof of Theorem~\ref{MT}}
We begin by introducing functionals on $\cB_\om$ analogous to those employed by Morris in~\cite{Morris}. More precisely,  for $t\in \mathbb R$ and $\omega \in \Omega$, let $\ell_\om^t \colon \cB_\om \to \mathbb C$ be given by 
\begin{equation}\label{deff}
\ell_\om^{t}(\varphi)=\int_{X_\om} e^{-it H(\omega, \cdot)}\varphi \, dm_\om, \quad \varphi \in \cB_\om.
\end{equation}
We note that $\ell_\om^{t} \in \cB_\om^*$ and 
\begin{equation}\label{uniform bound}
    \| \ell_\om^{t}\|_{\cB_\om^*}\le 1, \quad \text{for $\omega \in \Omega$ and $t\in \mathbb R$.}
\end{equation}
Indeed, for arbitrary $\varphi \in \cB_\om$ we have
\[
|\ell_\om^{t}(\varphi)| \le \int_{X_\om} |e^{-it H(\omega, \cdot)}\varphi|\, dm_\om=\int_{X_\om} |\varphi|\, dm_\om=\|\varphi\|_{L^1(m_\om)}\le \|\varphi\|_{\cB_\om},
\]
yielding the desired claim. We observe some additional facts.
\begin{lemma}\label{funct}
The following holds:
\begin{enumerate}
\item for $\omega \in \Omega$ and $t\in \mathbb R$, $\ell_\omega^{t}\neq 0$;
\item for $\omega \in \Omega$ and $t\in \mathbb R$, 
\[
(\cL_\omega^{it})^*(\ell_{\sigma \omega}^{t})=\ell_\omega^{t},
\]
where $\cL_\omega^\theta$ for $\omega \in \Omega$ and $\theta \in \mathbb C$ is given by~\eqref{tto}.
\end{enumerate}

\end{lemma}

\begin{proof}
We fix $\omega \in \Omega$ and $t\in \mathbb R$.   Since $\cB_\om$ is dense in $L^1(m_\om)$ (see~\cite[(V6)]{DGGV}), there exists $\varphi\in \cB_\om$ such that \[\|\varphi -e^{itH(\om, \cdot)}\|_{L^1(m_\om)}<1.\] Thus, 
 \[\begin{split}
 |\ell^t_\om(\varphi)-1|&\leq\int_{X_\om} \left|e^{-itH(\om,\cdot)}\varphi-1\right|\, dm_\om\\
&= \int_{X_\om} \left|e^{-itH(\om, \cdot)}\left(\varphi-e^{itH(\om, \cdot)}\right)\right|\, dm_\om\\ 
 &\leq  \|\varphi -e^{itH(\om, \cdot)}\|_{L^1(m_\om)}\\
 &<1,
 \end{split}\]
 which implies that $\ell^t_\om(\varphi)\neq 0$. We conclude that $\ell_\om^t\neq 0$, which is the first assertion of the lemma.

 In order to establish the second assertion, we take an arbitrary $\varphi \in \cB_\om$ and observe that 
\[\begin{split}
        \left(\cL_\om^{it}\right)^\ast(\ell^t_{\sig \om}) (\varphi)&= \ell^t_{\sig \om}(\cL_\om^{it}(\varphi))\\
        &=\int_{X_{\sig\om}} e^{-itH(\sig \om, \cdot )}\cL_\om^{it}(\varphi) \, dm_{\sig\om}\\
        &=\int_{X_{\sig\om}} e^{-itH(\sig \om , \cdot)}\cL_\om(e^{i t F(\om,\cdot)}\varphi) \, dm_{\sig\om}\\
        &=\int_{X_\om} e^{-itH(\sig(\om),T_\om (\cdot))} e^{i t F(\om, \cdot)}\varphi \, dm_\om\\
        &=\int_{X_\om} e^{-it H(\omega, \cdot)}\varphi \, dm_\om \\
        &=\ell_\om^t(\varphi),
    \end{split}\]
    where we used~\eqref{dual} and~\eqref{FH}. Thus, the second conclusion of the lemma is valid.
\end{proof} 


Let $\tilde{\cL}_\om^{\theta}$ be defined by~\eqref{tto} by replacing $F$ with $\tilde F$. Due to~\eqref{tildeF}, we have 
\begin{equation}\label{tildeL}
\tilde{\cL}_\om^{\theta}=e^{-\theta c_\om}\cL_\om^{\theta}, 
\end{equation}
where $c_\om:=\int_{X_\om} F(\om, \cdot)\, d\mu_\om$. Note that the pair $(\tilde F, \mathcal L)$ also admits a  random RPF triplet. Moreover, 
it is easy to relate the triples for $(F, \cL)$ and $(\tilde F, \cL)$.
That is, if $\theta\in U$ where $U$ is a neighborhood of $0\in \mathbb C$ and $(\lambda_\om^\theta, v_\om^\theta, \phi_\om^\theta)\in \mathbb C\times \cB_\om\times \cB_\om^*$, $\omega \in \Omega$ is the triplet corresponding to $(F, \cL)$, then $(\tilde{\lambda}_\om^\theta, v_\om^\theta, \phi_\om^\theta)$, $\omega \in \Omega$ is the triplet corresponding to $(\tilde F, \cL)$, where \begin{equation}\label{tildelambda}\tilde{\lambda}_\om^\theta=e^{-\theta c_\om}\lambda_\om^\theta.\end{equation}

\begin{lemma}\label{varvanishes}
We have $\Sigma^2=0$, where $\Sigma^2$ is given by~\eqref{variance}.
\end{lemma}

\begin{proof}
For $t\in \mathbb R$ sufficiently close to $0$, it follows from~\eqref{newadd} that 
\begin{equation}\label{aux1}
\left \|\cL_\omega^{it, (n)}\varphi-\left (\prod_{k=0}^{n-1}\lambda_{\sigma^k \omega}^{it}\right )\phi_\om^{it}(\varphi)v_{\sigma^n \omega}^{it}\right \|_{\cB_{\sigma^n\om}}\le cr^n\|\varphi \|_{\cB_\om},
\end{equation}
for $\mathbb P$-a.e. $\omega \in \Omega$, $n\in \mathbb N$ and $\varphi \in \cB_\om$.
By~\eqref{uniform bound} and~\eqref{aux1} we have that 
\[
\begin{split}
&\left | ((\cL_\omega^{it, (n)})^* \ell_{\sigma^n \omega}^{t})(\varphi)-\left (\prod_{k=0}^{n-1}\lambda_{\sigma^k \omega}^{it}\right )\phi_\om^{it}(\varphi)\ell_{\sigma^n\omega}^{t}(v_{\sigma^n \omega}^{it})\right | \\
&=\left |\ell_{\sigma^n \om}^{t} \left (\cL_{ \omega}^{it, (n)}\varphi-\left (\prod_{k=0}^{n-1}\lambda_{\sigma^k \omega}^{it}\right )\phi_\om^{it}(\varphi)v_{\sigma^n \om}^{it}\right )\right | \\
&\le \left \|\cL_{\om}^{it, (n)}\varphi-\left (\prod_{k=0}^{n-1}\lambda_{\sigma^k \omega}^{it}\right )\phi_\om^{it}(\varphi)v_{\sigma^n \om}^{it}\right \|_{\cB_{\sigma^n\om}}\\
&\le cr^n \|\varphi\|_{\cB_\om},
\end{split}
\]
for $\mathbb P$-a.e. $\omega \in \Omega$, $n\in \mathbb N$ and $\varphi \in \cB_\om$. This together with the second assertion of Lemma~\ref{funct} gives that
\begin{equation}\label{928}
\ell_\om^{t}=\lim_{n\to \infty}\left (\prod_{k=0}^{n-1}\lambda_{\sigma^k \omega}^{it}\right )\ell_{\sigma^n\omega}^{t}(v_{\sigma^n \omega}^{it})\phi_\om^{it},
\end{equation}
for $\mathbb P$-a.e. $\omega \in \Omega$.

Suppose that $\Sigma^2>0$. 
From the last requirement in Definition~\ref{RPF}, it follows that there exist $c, \tilde \delta>0$  such that for $\mathbb P$-a.e. $\om \in \Om$, $t\in [-\tilde \delta, \tilde \delta]$ and $n\ge n_0=n_0(\om)$,
\begin{equation}\label{above}
\left |\prod_{k=0}^{n-1}\tilde \lambda_{\sigma^k \om}^{it}\right |=\left |\prod_{k=0}^{n-1} \lambda_{\sigma^k \om}^{it}\right |
\le e^{-cnt^2\Sigma^2},
\end{equation}
where we have also taken into account~\eqref{tildelambda}.
Observe that the right-hand side in~\eqref{above} goes to $0$ when $n\to \infty$ for every $t\in [-\tilde \delta, \tilde \delta]\setminus \{0\}$. Fix such $t$. We have
\[
\lim_{n\to \infty}\prod_{k=0}^{n-1}\lambda_{\sigma^k \omega}^{it}=0 \quad \text{for $\mathbb P$-a.e. $\omega \in \Omega$,}
\]
which together with~\eqref{930} and~\eqref{928} gives $\ell_\om^{it}=0$ for $\mathbb P$-a.e. $\omega \in \Omega$, contradicting the first assertion of Lemma~\ref{funct}. Therefore, $\Sigma^2=0$.
\end{proof}

Next, we consider the martingale decomposition associated with $\tilde F$ constructed in~\cite[Section 4]{DH1}, which we briefly outline for the sake of completeness. 
For $\omega \in \Omega$, set 
\begin{equation}\label{chiom}
\chi_\om:=\sum_{n=0}^\infty L_{\sigma^{-n}\omega}^{(n)}\tilde F(\sigma^{-n}\omega, \cdot).
\end{equation}
where $L_{\omega}^{(n)}$ is given by
\begin{equation}\label{eq: expression iterate L}
    L_\om^{(n)} \varphi=\frac{\cL^{(n)}_\om (\varphi v_\om)}{v_{\sigma^n \om}}, \quad \varphi \in \cB_\om.
\end{equation}
We note that $L_\omega=L_\omega^{(1)}\colon \cB_\omega \to \cB_{\sigma \omega}$ is a  bounded linear operator for $\mathbb P$-a.e. $\omega \in \Omega$. This easily follows from the assumption that $\cL_\omega \colon \cB_\omega \to \cB_{\sigma \omega}$ is bounded together with~\cite[(V3), (V7), (V8)]{DGGV} and \eqref{essinf}.
The same type of reasoning together with~\eqref{newadd} (applied to $\theta=0$) gives that there exist $D, \lambda>0$ such that 
\begin{equation}\label{tu}
    \|L_\omega^{(n)}\varphi\|_{\cB_{\sigma^n \omega}}\le De^{-\lambda n}\|\varphi\|_{\cB_\omega},
\end{equation}
for $\mathbb P$-a.e. $\omega \in \Omega$, $n\in \mathbb N$ and $\varphi \in \cB_\omega$ with $\int_{X_\omega}\varphi\, d\mu_\omega=0$. Indeed, \eqref{newadd} for $\theta=0$ gives that 
\[
\|\mathcal L_\omega^{(n)}\varphi\|_{\cB_{\sigma^n \omega}}\le cr^n \|\varphi\|_{\cB_\omega}, 
\]
for $\mathbb P$-a.e. $\omega \in \Omega$, $n\in \mathbb N$ and $\varphi \in \cB_\omega$ with $\int_{X_\omega}\varphi \, dm_\omega=0$. Moreover, as we have a uniform upper bound for $\|v_\omega\|_{\cB_\omega}$ and $\|1/v_\omega \|_{\cB_\omega}$, which is independent of $\omega$ on a set of full probability (owing to~\eqref{essinf} and~\eqref{930}), we obtain~\eqref{tu}.

By~\eqref{tu}, $\chi_\om\in \cB_\om$ for $\mathbb P$-a.e. $\om \in \Om$. Moreover, there is $\bar D>0$ such that 
\begin{equation}\label{upperbound-2}
\|\chi_\om\|_{\cB_\om}\le \tilde D, \quad \text{for $\mathbb P$-a.e. $\omega \in \Omega$.}
\end{equation}
For $\om \in \Om$, let 
\begin{equation}\label{mdec}
\pi_\om:=\tilde F(\om, \cdot)+\chi_\om-\chi_{\sigma \om}\circ T_\om.
\end{equation}
Then, $L_\om(\pi_\om)=0$ for $\mathbb P$-a.e. $\om\in \Om$, and consequently 
\begin{equation}\label{rmd}
\mathbb E_\om[\pi_{\sigma^n \om}\circ T_\om^n\rvert (T_\om^{(n+1)})^{-1}(\mathcal G)]=0,
\end{equation}
for $\mathbb P$-a.e. $\om \in \Om$ and $n\in \mathbb N$, where the left-hand side in the above equality denotes the conditional expectation with respect to $\mu_\om$. We refer to~\cite[Lemmas 3 and 4]{DH1} for details.

\begin{lemma}\label{MartLem}
We have $\pi_\om=0$ for $\mathbb P$-a.e. $\om \in \Om$.
\end{lemma}

\begin{proof}
Let $S_n \pi$ be defined by~\eqref{birksums} by replacing $\tilde F$ with $\pi$ given by $\pi(\om, x):=\pi_\om(x)$, $(\om, x)\in \cX$. By~\eqref{upperbound-2}, \eqref{mdec} and~\cite[(V3)]{DGGV}, there exists $\bar D>0$ such that
\begin{equation}\label{948}
\|S_n \tilde F(\om, \cdot)-S_n \pi(\om, \cdot)\|_{L^2(\mu_\om)}\le \bar  D,\quad \text{for $\mathbb P$-a.e. $\om \in \Om$.}
\end{equation}
Due to the Lemma~\ref{varvanishes}, we have $\Sigma^2=0$. Hence, \eqref{asymvar} and~\eqref{948} give
\begin{equation}\label{957}
\lim_{n\to \infty}\frac 1 n \|S_n \pi(\om, \cdot)\|_{L^2(\mu_\om)}^2=0, \quad \text{for $\mathbb P$-a.e. $\om \in \Om$.}
\end{equation}
On the other hand, \eqref{rmd} implies
\[
\|S_n \pi(\om, \cdot)\|_{L^2(\mu_\om)}^2=\sum_{k=0}^{n-1}\|\pi_{\sigma^k \om}\|_{L^2(\mu_{\sigma^k \om})}^2,
\]
and consequently
\begin{equation}\label{958}
\lim_{n\to \infty}\frac 1 n \|S_n \pi(\om, \cdot)\|_{L^2(\mu_\om)}^2 = \int_\Omega \|\pi_\om\|_{L^2(\mu_\om)}^2\, d\mathbb P(\omega),
\end{equation}
for $\mathbb P$-a.e. $\om \in \Om$. The conclusion of the lemma now follows easily from~\eqref{957} and~\eqref{958}.
\end{proof}
From~\eqref{mdec} and the previous lemma, we conclude that  
\begin{equation}\label{143}
\tilde F(\om, \cdot)=\chi_{\sigma \om}\circ T_\om -\chi_\om, \quad \text{for $\mathbb P$-a.e. $\om \in \Om$.}
\end{equation}
Setting $\hat H(\om, \cdot):=H(\om, \cdot)-\chi_\om(\cdot)$, from~\eqref{FH},~\eqref{tildeF} and~\eqref{143} it follows that 
\[
c_\om=\int_{X_{\om}}F(\om, \cdot)\, d\mu_\om=\hat H(\sigma \om, T_\om(\cdot))-\hat H(\om, \cdot),
\]
for $\mathbb P$-a.e. $\om \in \Om$. 

For $\om \in \Om$ and $t\in \mathbb R$, let $\bar{\cL}_\om^{it}\colon \cB_\omega \to \cB_{\sigma \omega}$ be a bounded linear operator defined by 
\[
\bar{\cL}_\om^{it}\varphi:=\cL_\om(e^{itc_\om}\varphi)=e^{it c_\om}\cL_\om \varphi, \quad \varphi \in \cB_\om.
\]
Note that $\bar{\cL}_\om^{it}$ is defined as $\cL_\om^{it}$ by replacing $F_\om$ with $c_\om$.
Furthermore, we set 
\[
\bar{\cL}_\om^{it, (n)}:=\bar{\cL}_{\sigma^{n-1}\om}^{it}\circ \ldots \circ \bar{\cL}_{\sigma \om}^{it}\circ \bar{\cL}_\om^{it}, \quad \om \in \Om, \ n\in \mathbb N.
\]
Clearly,
\[
\bar{\cL}_\om^{it, (n)}=e^{it \sum_{j=0}^{n-1}c_{\sigma^j \om}}\cL_\om^{(n)}, \quad \om \in \Om, \ n\in \mathbb N.
\]
Then \eqref{930}, \eqref{newadd}, and~\cite[(V3)]{DGGV} imply that 
\[
\begin{split}
&\left \|\bar{\cL}_\om^{it, (n)}\varphi-e^{it \sum_{j=0}^{n-1}c_{\sigma^j \om }}\left(\int_{X_\om} \varphi\, dm_\om\right )v_{\sigma^n \om}\right \|_{\cB_{\sigma^n\om}} \\
&=\left \|e^{it \sum_{j=0}^{n-1}c_{\sigma^j \om }}\cL_\om^{(n)} \left (\varphi-\left (\int_{X_\om} \varphi\, dm_\om\right )v_\om\right)\right \|_{\cB_{\sigma^n\om}}\\
&=\left \|\cL_\om^{(n)} \left (\varphi-\left (\int_{X_\om} \varphi\, dm_\om\right )v_\om\right) \right  \|_{\cB_{\sigma^n\om}}\\
&\le cr^n\left \|\varphi-\left (\int_{X_\om} \varphi\, dm_\om\right )v_\om\right \|_{\cB_\om}\\
&\le \bar  c r^n\|\varphi\|_{\cB_\om},
\end{split}
\]
for $\mathbb P$-a.e. $\om \in \Om$, $n\in \mathbb N$, $t\in \mathbb R$ and $\varphi \in \cB_\omega$, where $\bar c>0$ is a constant independent of these variables.  Consequently, 
\begin{equation}\label{526}
\left \| (\bar{\cL}_\om^{it, (n)})^*\phi -e^{it \sum_{j=0}^{n-1}c_{\sigma^j \om}}\phi(v_{\sigma^n \om})m_\om\right \|_{\cB^*_{\om}}\le \bar c r^n\|\phi\|_{\cB_{\sigma^n\om}^*},
\end{equation}
for $\mathbb P$-a.e. $\om \in \Om$, $n\in \mathbb N$, $t\in \mathbb R$ and $\phi \in \cB_{\sigma^n\om}^*$, where again we view $m_\om$ as an element $\cB_\om^*$.

For $\om \in \Om$ and $t\in \mathbb R$, let $\bar \ell_\om^{t}\in \cB_\om^*$ be defined as $\ell_\om^t$ (see~\eqref{deff}), replacing $H(\om, \cdot)$ with $\hat H(\om, \cdot)$. The same arguments as in the proof of Lemma~\ref{funct} yield $\bar \ell_\om^t \neq 0$, $\|\bar \ell_\om^{it}\|_{\cB_\om^*}\le 1$
and
\[
(\bar{\cL}_\om^{it})^* \bar \ell_{\sigma \om}^t=\bar \ell_\om^t, \quad \text{for $\mathbb P$-a.e. $\om \in \Om$ and $t\in \mathbb R$.}
\]
Hence, applying~\eqref{526} to $\phi=\bar \ell_{\sigma^n \om}^t$ we have 
\[
\left \| \bar{\ell}_\om^t-e^{it \sum_{j=0}^{n-1}c_{\sigma^j \om}}\bar \ell_{\sigma^n \om}^t(v_{\sigma^n \om})m_\om \right \|_{\cB^*_\omega}\le \bar c r^n,
\]
for $\mathbb P$-a.e. $\om \in \Om$, $n\in \mathbb N$ and $t\in \mathbb R$. Therefore, 
\[
\bar{\ell}_\om^t=\lim_{n\to \infty}e^{it \sum_{j=0}^{n-1}c_{\sigma^j \om}}\bar \ell_{\sigma^n \om}^t(v_{\sigma^n \om})m_\om,
\]
for $\mathbb P$-a.e. $\om \in \Om$, $n\in \mathbb N$ and $t\in \mathbb R$. In particular, 
\begin{equation}\label{553}
\lim_{n\to \infty}e^{it \sum_{j=0}^{n-1}c_{\sigma^j \om}}\bar \ell_{\sigma^n \om}^t(v_{\sigma^n \om})=\bar{\ell}_\om^{it}(1)=\int_{X_\om} e^{-it \hat H(\om, \cdot)}\, dm_\om,
\end{equation}
and 
\begin{equation}\label{600}
\left (\lim_{n\to \infty}e^{it \sum_{j=0}^{n-1}c_{\sigma^j \om}}\bar \ell_{\sigma^n \om}^t(v_{\sigma^n \om}) \right)\int_{X_\om} e^{it \hat H(\om, \cdot)}\, dm_\om=\bar \ell_\om^t(e^{it \hat H(\om, \cdot)})=1,
\end{equation}
for $\mathbb P$-a.e. $\om \in \Om$ and $t\in \mathbb R$. From~\eqref{553} and~\eqref{600} we have 
\[
\int_{X_\om} e^{-it \hat H(\om, \cdot)}\, dm_\om=\frac{1}{\int_{X_\om} e^{it \hat H(\om, \cdot)}\, dm_\om},
\]
and therefore 
\[
\left |\int_{X_\om} e^{it \hat H(\om, \cdot)}\, dm_\om\right |=1, \quad \text{for $\mathbb P$-a.e. $\om \in \Om$ and $t\in \mathbb R$.}
\]
On the other hand, since 
\[
\left |\int_{X_\om} e^{it \hat H(\om, \cdot)}\, dm\right |^2=1-2\int_{X_\om\times X_\om}\sin^2\left (\frac{t \hat H(\om, x)-t\hat H(\om, y)}{2}\right)\, d (m_\om\times m_\om)(x, y),
\]
 we have that for $\mathbb P$-a.e. $\om \in \Omega$ and $t\in \mathbb R$ 
 \[
 t\hat H(\om, \cdot)=\alpha_\om^t+2\pi k_\om^t(\cdot),
 \]
 for some $\alpha_\om^t\in [0,2\pi)$ and $k_\om^t \colon X\to \mathbb Z$.  

 Take an arbitrary $t\in \mathbb R\setminus \{0\}$. We claim that $k_\om^t$ is a constant function for $\mathbb P$-a.e. $\om \in \Om$. Suppose that $k_\om^t(x)-k_\om^t(y)\neq 0$ for some $x, y\in X$, $x\neq y$. Let $s:=t\sqrt 2$. Then
 \[
 \hat H(\omega, x)-\hat H(\om, y)=\frac{2\pi}{t}(k_\om^t(x)-k_\om^s(y))=\frac{2\pi}{s}(k_\om^s(x)-k_\om^s(y)),
 \]
 yielding $\sqrt 2=s/t\in \mathbb Q$. Therefore, $k_\om^t$  is a constant function for $\mathbb P$-a.e. $\om \in \Om$. Consequently, $\hat H(\om, \cdot)$ is a constant for $\mathbb P$-a.e. $\om \in \Om$. Since $\chi_\om \in \cB_\omega$, we immediately get that 
 \[
 H(\om, \cdot)=\hat H(\om, \cdot)+\chi_\om \in \cB_\om, \quad \text{for $\mathbb P$-a.e. $\om \in \Om$.}
 \]
 As $\int_{X_\om} \chi_\om \, d\mu_\om=0$, from the above it yields that 
 \[
\hat  H(\om, \cdot)=\int_{X_\om} H(\om, \cdot)\, d\mu_\om,
 \]
 which together with \eqref{eq: expression iterate L} and~\eqref{chiom} gives~\eqref{sumH}. This completes the proof of Theorem~\ref{MT}.

\section{Sequential dynamics}\label{SDS}
In this section, we will present our first main result in the sequential setting.

\subsection{Setting}
Let $(X_j, \mathcal G_j, m_j)$, $j\in \mathbb N_0:=\mathbb N\cup \{0\}$ be a sequence of probability spaces endowed with notions of variations $\var_j \colon L^1(X_j, m_j)\to [0, +\infty]$ that satisfy the conditions $(V1)-(V7)$ of~\cite[p.5]{DolgHaf}. For $j\in \mathbb N_0$,  we define 
\[
\cB_j:=\{\varphi\in L^1(X_j, m_j): \ \var_j(\varphi)<+\infty\}.
\]
Then, each $\cB_j$ is a Banach space with respect to the norm 
\[
\|\varphi\|_{\cB_j}:=\|\varphi\|_{L^1(m_j)}+\var_j (\varphi).
\]
Let $T_j:X_j\to X_{j+1},$ $j\in \mathbb N_0$ be a sequence of measurable maps such that
\begin{equation}\label{V8}
    \sup_j\sup_{h:\, \var_{j+1}(h)\leq 1}\var_{j}(h\circ T_j)<\infty.
    \end{equation}
We also assume that the maps are absolutely continuous,  that is, $(T_j)_*m_j\ll m_{j+1}$ for each $j\in \mathbb N_0$. Let $\cL_j\colon L^1(X_j, m_j)\to L^1(X_{j+1}, m_{j+1})$ denote the transfer operator associated with $T_j$ with respect to the measures $m_j$ and $m_{j+1}$ characterized by the following duality relation:
\begin{equation}\label{sdual}
\int_{X_{j+1}} (\cL_j \varphi)\psi \, dm_{j+1}=\int_{X_j}\varphi \cdot (\psi \circ T_j)\, dm_j, 
\end{equation}
for $\varphi \in L^1(X_j, m_j)$ and $\psi \in L^\infty(X_{j+1}, m_{j+1})$. For $j\in \mathbb N_0$ and $n\in \mathbb N$, set 
\[
\cL_j^{(n)}:=\cL_{j+n-1}\circ \ldots \circ \cL_{j+1}\circ \cL_j,
\]
which is the transfer operator corresponding to 
\[
T_j^{(n)}:=T_{j+n-1}\circ \ldots \circ T_{j+1} \circ T_j.
\]
\begin{definition}\label{admissS}
We say that the sequence of transfer operators $(\cL_j)_{j\in \mathbb N_0}$  is \emph{admissible} if the following holds:
\begin{itemize}
\item there exists $K>0$ such that 
\[
\|\cL_j \varphi\|_{\cB_{j+1}}\le K\|\varphi\|_{\cB_j} \quad \text{for $j\in \mathbb N_0$ and $\varphi \in \cB_j$;}
\]
\item there are $N\in \mathbb N$, $\alpha^N \in (0, 1)$ and $\beta^N>0$ such that for every $j\in \mathbb N_0$ and $\varphi \in \cB_j$,
\[
\|\cL_j^{(N)}\varphi \|_{\cB_{j+N}}\le \alpha^N \|\varphi\|_{\cB_j}+\beta^N \|\varphi\|_{L^1(X_j, m_j)};
\]
\item for each $a>0$, there are $c=c(a)>0$ and $n_0=n_0(a)\in \mathbb N$ such that 
\[
\essinf \cL_j^{(n)}\varphi\ge c\|\varphi\|_{L^1(X_j, m_j)}, \quad \text{for $j\in \mathbb N_0$, $n\ge n_0$ and every $\varphi\in C_{j, a}$,}
\]
where
\[
C_{j, a}:=\left \{\varphi\in \cB_j:  \ \varphi\ge 0 \ \text{and} \ \var_j(\varphi)\le a\int_{X_j} \varphi\, dm_j \right \}.
\]
\end{itemize}
\end{definition}

\begin{remark}
We refer to~\cite[Section 4]{DolgHaf} for explicit examples of admissible sequences of transfer operators.

\end{remark}

\begin{remark}\label{seqim}
Let $(\cL_j)_{j\in \mathbb N_0}$ be any admissible sequence of transfer operators. By~\cite[Theorem 2.4]{DolgHaf} there exists a sequence $(v_j)_{j\in \mathbb N_0}$ of nonnegative maps $v_j \colon X_j \to \mathbb R$ with $v_j\in \cB_j$ such that the following holds:
\begin{itemize}
\item for $j\in \mathbb N_0$, $\int_{X_j}v_j\, dm_j=1$;
\item  $\sup_{j\in \mathbb N_0}\|v_j\|_{\cB_j}<+\infty$;
\item there is $c>0$ such that $\essinf v_j\ge c$ for $j\in \mathbb N_0$;
\item for $j\in \mathbb N_0$, $\cL_j v_j=v_{j+1}$;
\item there are $C, \lambda>0$ such that 
\begin{equation}\label{835eq}
\left \|\cL_j^{(n)}\varphi-\left (\int_X \varphi \, dm_j\right ) v_{j+n}\right \|_{\cB_{j+n}}\le Ce^{-\lambda n}\|\varphi\|_{\cB_j},
\end{equation}
for $j\in \bbN_0$, $n\in \bbN$ and $\varphi \in \cB_j$.
\end{itemize}
For $j\in \mathbb N_0$, let $\mu_j$ be the probability measure on $X_j$ given by $d\mu_j=v_j\, dm_j$. Then $T_j^*\mu_j=\mu_{j+1}$ for $j\in \mathbb N_0$.

 Measures $\mu_j$ can be viewed as sequential counterparts of the random measures $\mu_\omega$ discussed in Section \ref{sec: setting random}.
However, in contrast to the random case, the sequences of measures $\mu_j$, $j\in \mathbb N_0$ with the above properties are not unique. On the other hand, the following is true: Let $\tilde \mu_j$, $j\in \mathbb N_0$ be any sequence where $\tilde \mu_j$ is a probability measure on $X_j$ with $d\tilde \mu_j=\tilde v_j\, dm_j$ such that $T_j^*\tilde \mu_j=\tilde \mu_{j+1}$ for $j\in \mathbb N_0$, then
\[
\lim_{j\to \infty}\|v_j-\tilde v_j\|_{L^1(X_j, m_j)}=0.
\]

\end{remark}

\subsection{Perturbation theory}
Throughout this subsection, we fix an admissible sequence of transfer operators $(\cL_j)_{j\in \mathbb N_0}$. Let $F=(F_j)_{j\in \mathbb N_0}$ be a sequence of functions $F_j\colon X_j \to \mathbb R$ satisfying the following properties:
\begin{itemize}
\item \begin{equation}\label{sobs1} F_j\in \cB_j \quad \text{for $j\in \mathbb N_0$;} \end{equation}
\item \begin{equation}\label{sobs2}\sup_{j\in \mathbb N_0}\|F_j\|_{\cB_j}<+\infty.\end{equation}
\end{itemize}
For $\theta\in \mathbb C$ and $j\in \mathbb N_0$, let $\cL_j^\theta$ be a linear operator $\cL_j^\theta \colon \cB_j \to \cB_{j+1}$ defined by 
\begin{equation}\label{ttop}
\cL_j^\theta \varphi:=\cL_j(e^{\theta F_j}\varphi), \quad \varphi \in \cB_j.
\end{equation}
Then $\cL_j^\theta$ is a bounded linear operator (see~\cite[Lemma 7.2]{DolgHaf}) for each $j\in \mathbb N_0$ and $\theta \in \mathbb C$. 

The following result is obtained in~\cite[Theorem 7.3]{DolgHaf}.
\begin{theorem}\label{PTS}
There exists a neighborhood $U$ of $0$ in $\mathbb C$, and for each $\theta \in U$ there is a triplet $(\lambda_j^\theta, v_j^\theta, \phi_j^\theta)\in \mathbb C\times \cB_j \times \cB_j^*$, $j \in \mathbb N_0$ such that the following holds:
\begin{enumerate}
\item for $j\in \mathbb N_0$ and for all $\theta \in U$, 
\[
\cL_j^\theta v_j^\theta=\lambda_j^\theta v_{j+1}^\theta, \quad (\cL_j^\theta)^*\phi_{j+1}^\theta=\lambda_j^\theta \phi_j^\theta \quad \text{and} \quad \phi_j^\theta(v_j^\theta)=1;
\]
\item for $j\in \mathbb N_0$, $\lambda_j^0=1$, $v_j^0=v_j$, and $\phi_j^0=m_j$, where $m_j$ is identified with the functional in $\cB_j^*$ given by $\cB_j\ni\varphi \mapsto \int_{X_j} \varphi \, dm_j$, and $v_j$, $j\in \mathbb N_0$ are as in Remark~\ref{seqim};
\item for $j\in \mathbb N_0$, the maps $\theta \mapsto \lambda_j^\theta$, $\theta \mapsto v_j^\theta$, and $\theta \mapsto \phi_j^\theta$ are analytic on $U$;
\item there is $C>0$ such that for $j\in \mathbb N_0$ and  $\theta \in U$,
\begin{equation}\label{930s}
\max \{|\lambda_j^\theta|, \|v_j^\theta\|_{\cB_j}, \|\phi_j^\theta\|_{\cB_j^*}\}\le C;
\end{equation}
\item there are $c>0$ and $r\in (0, 1)$ such that for $j\in \mathbb N_0$ and  $\theta \in U$,
\begin{equation}\label{newadds}
\left \| \mathcal L_j^{\theta, (n)}\varphi-\left (\prod_{i=j}^{j+n-1}\lambda_i^\theta \right )\phi_j^{\theta}(\varphi)v_{j+n}^\theta \right \|_{\cB_{j+n}} \le cr^n\|\varphi\|_{\cB_j},
\end{equation}
for every $\varphi \in \cB_j$ and $n\in \mathbb N$, where 
\[
\cL_j^{\theta, (n)}:=\cL_{j+n-1}^{\theta} \circ \ldots \circ \cL_{j+1}^\theta \circ \cL_j^\theta. 
\]
\end{enumerate}

\end{theorem}

\begin{remark}
Triplets $(\lambda_j^\theta, v_j^\theta, \phi_j^\theta)$ can be regarded as sequential counterparts to random RPF triplets introduced in Definition~\ref{RPF}.
\end{remark}

\subsection{Statement of the second main result}
The following is the second main result of our paper.
\begin{theorem}\label{MT2}
Let $(\cL_j)_{j\in \mathbb N_0}$ be an admissible sequence of
transfer operators. Furthermore, let $F=(F_j)_{j\in \mathbb N_0}$ be a sequence of maps $F_j\colon X_j\to \mathbb R$ satisfying~\eqref{sobs1} and~\eqref{sobs2} and with the property that for each $j\in \mathbb N_0$,
\begin{equation}\label{sobsdec}
F_j=H_{j+1}\circ T_j-H_j, 
\end{equation}
for some  measurable maps $H_j\colon X_j\to \mathbb R$.

Then 
\begin{equation}\label{secdec}
F_j=\int_{X_0}F_j\circ T_0^{(j)}\, dm_0+U_{j+1}\circ T_j-U_j+\mathcal M_j \quad \text{for $j\in \mathbb N_0$,}
\end{equation}
where\begin{enumerate}
\item $U_j, \mathcal M_j\in \cB_j$ for each $j\in \mathbb N_0$. Moreover, $\sup_{j\in \mathbb N_0}\|U_j\|_{\cB_j}<+\infty$, \\ $\int_{X_0} U_j \circ T_0^{(j)}\,dm_0=0$ and $\sup_{j\in \mathbb N_0}\|\mathcal M_j\|_{\cB_j}<+\infty$;
\item for $j\in \mathbb N_0$,  \[\mathbb E_{m_0}[\mathcal M_j\circ T_0^{(j)}\rvert (T_0^{(j+1)})^{-1}\mathcal G_{j+1}]=0. \] Moreover, $\sum_{j=0}^\infty \mathcal M_j\circ  T_0^{(j)}$ converges $m_0$-a.s. and in $L^p(m_0)$ for every finite $p\geq 1$.  
\item for every $n\in \bbN_0$ and a finite $p\geq1$ we have $\int_{X_0} |H_n \circ T_0^{(n)}|^p\,  dm_0<\infty$, $\sup_j\int_{X_0}|H_j\circ T_0^{(j)}-m_0(H_j\circ T_0^{(j)})|^pdm_0<\infty$
and
\begin{equation}\label{UHq}
H_n=\int_{X_0}H_n \circ T_0^{(n)}\,dm_0+U_n-\sum_{k=n}^\infty \mathcal M_k\circ T_n^{(k-n)}.
\end{equation}
In particular, with $q_n=\int_{X_0}H_n\circ T_0^{(n)}\,dm_0$ we have
\begin{equation}\label{LimExp}
\lim_{n\to\infty}\left((H_n-U_n)\circ T_0^{(n)}-q_n\right)=0,\,\,m_0-\text{a.e}
\end{equation}
and in $L^p(m_0)$ for all finite $p\geq 1$.
\end{enumerate}
\end{theorem}
 In \cite[Lemma 6.3]{DolgHaf} it is shown that $F_j$ always has a representation of the form \eqref{secdec} except that, in general, $\sum_{j=0}^\infty \mathcal M_j\circ T_0^j$ might not converge. In the proof of Theorem \ref{MT2} we will show that under \eqref{sobsdec} the latter series indeed converges. By \cite[Theorem 6.5]{DolgHaf}, this is equivalent to $\sup_n\text{Var}_{m_0}(S_{0,n}F)<\infty$ where $S_{0,n}F$ is defined in \eqref{eq: def Sj,nG} below and $\text{Var}_{m_0}(\cdot)$ is the variance with respect to the measure $m_0$. We thus get the following counterpart of Corollary \ref{Cor1}. 
\begin{corollary}\label{Cor2}
 $\sup_n\text{Var}_{m_0}(S_{0,n}F)<\infty$ if and only if 
 $$
F_j-\int_{X_0} F_j\circ T_0^j dm_0=H_{j+1}\circ T_j-H_j, 
 $$
 for $j\in \mathbb N_0$ and some measurable function $H_j$. In that case  $H_n=U_n-\sum_{k=n}^\infty \mathcal M_k\circ T_n^{(k-n)}$. The same conclusion holds without centering  $F_j$ (but with \eqref{UHq} instead of the latter formula for $H_n$).
\end{corollary}
 
\begin{remark}\label{Rem0}
Having in mind the (above) discussion proceeding Theorem \ref{MT2},  at first glance it might seem more natural to replace \eqref{sobsdec} by representations of the form
 \begin{equation}\label{Cob}
 F_j=a_j+M_j+H_{j+1}\circ T_j-H_j \quad j\in \bbN_0,  
 \end{equation}
 for some constants $a_j$ and measurable functions $H_j,M_j:X_j\to\bbR$ such that the series $\sum_{j=0}^\infty M_j\circ T_0^{(j)}$ converges almost surely. However, we can write $a_j=A_j-A_{j-1}$ with $A_j=\sum_{k=0}^{j}a_k$ and $M_j=V_{j+1}\circ T_j-V_j$ where $V_j=-\sum_{k=j}^\infty M_k\circ T_j^{(k-j)}$. Let $\bar H_j=H_j+V_j+A_{j-1}$. Then by \eqref{Cob},
 $$
F_j=\bar H_{j+1}\circ T_j-\bar H_j \quad j\in \bbN_0.
 $$
 Therefore, we can reduce \eqref{Cob} to \eqref{sobsdec}. Note that the same thing can be done in \eqref{secdec}, that is, one can absorb $\int_{X_0}F_j\circ T_0^{(j)}dm_0$ and $\cM_j$ inside the coboundary part; however, this ruins the regularity of the coboundary $U_n$ (that is, the new coboundary does not have a bounded variation).
\end{remark}
\begin{remark}\label{Rem}
We note that, in contrast to Theorem \ref{MT}, the coboundary part $H_j$, $j\in \mathbb N_0$ in~\eqref{sobsdec} are not necessarily functions in $\cB_j$. For example, assume that $\cB_j=\cB$ for $j\in \mathbb N_0$, and take a sequence of functions $(F_j)_{j\in \mathbb N_0}$ in $\cB$  such that the series $\sum_{k=0}^\infty F_k\circ T_0^{(k)}$ converges $m_0$-a.s. but not in $\cB$. Set 
\[
H_j=-\sum_{k=j}^\infty F_k\circ T_j^{(k-j)}, \quad j\in \mathbb N_0.
\] 
Then each $H_j$ is a measurable function which does not belong to $\cB$ and
$$
F_j=H_{j+1}\circ T_j-H_j, \quad j\in \mathbb N_0.
$$
In particular, \eqref{sobsdec} holds, which implies that Theorem~\ref{MT2} is applicable. Consequently, the maps $F_j$ can be written as in~\eqref{secdec}.
Since $\lim\limits_{n\to \infty}H_n\circ T_0^{(n)}=0$ $m_0$-a.s., from~\eqref{LimExp} we get that
\[
\lim_{n\to \infty}(U_n \circ T_0^{(n)}-q_n)=0 \quad \text{$m_0$-a.s.,}
\]
for some sequence of numbers $(q_n)_{n\in \bbN_0}\subset \bbR$.

In order to provide an explicit example, we take $X_j=[0, 1]$, $m_j=m$ where $m$ denotes the Lebesgue measure and $\var_j=\var$, where $\var (\varphi)=\int_0^1|\varphi'(x)|\, dx$.
Moreover, let $T_j=T$ for $j\in \mathbb N_0$, where $T\colon [0, 1] \to [0, 1]$ is the doubling map defined by $T(x)=2x (mod \ 1)$. Furthermore, let $f\colon [0, 1]\to \bbR$ be an arbitrary $C^1$-function such that  $\int_{0}^1 f\, dm=0$ and $\int_{0}^1 f'\, dm\not=0$. For $j\in \mathbb N_0$, let $F_j\colon [0, 1] \to \bbR$ be given by 
 \[
 F_j(x)=2^{-j}f(x), \quad x\in [0, 1].
 \]
 Clearly, $\sum_{k=0}^\infty F_k \circ T^{k}$ converges everywhere on $[0, 1]$. On the other hand, the series $\sum_{k=0}^\infty F_k \circ T^{k}$ does not converge in $\cB$ as 
\[
\var \left(\sum_{k=0}^{n-1} F_k \circ T^k\right)=\int_0^1 \left|\left(\sum_{k=0}^{n-1} F_k \circ T^{k}\right)'\right|\, dm=\int_0^1\left| \sum_{k=0}^{n-1}f'\circ T^k\right|\, dm,
\]
for every $n\in \mathbb N$, and in addition, 
\[
\lim_{n\to \infty}\frac 1 n \int_0^1\left|\sum_{k=0}^{n-1}f'\circ T^k\right|\, dm=\left|\int_0^1f'\, dm\right|\neq 0,
\]
due to Birkhoff's ergodic theorem (recall that $m$ is ergodic for $T$).

Finally, let us note that the functions $U_j$ in  \eqref{secdec} are not unique even under the restriction that $\sup_j\|U_j\|_{\cB_j}<\infty$. To illustrate this, let $X$, $T$, $m$ and $\var$ be as above. Set $T_j=T$ for $j\in \mathbb Z$. Then~\eqref{835eq} holds for every $j\in \mathbb Z$, $n\in \mathbb N$ with $v_n=1$ and $m_n=m$ for $n\in \mathbb Z$,  where $\cL_j^{(n)}=\cL^n$ and $\cL$ is the transfer operator associated with $T$. Take an arbitrary two-sided sequence $(F_j)_{j\in \mathbb Z}$ in $\cB$ such that $\int_XF_j \, dm=0$ and $\sup_{j\in \mathbb Z}\|F_j\|_{\cB}<+\infty$.  In this case,  the proof of \cite[Lemma 6.3]{DolgHaf} shows that~\eqref{secdec} holds with   $U_j=\sum_{k=1}^{j}\mathcal \cL^k F_{j-k}$, $j\in \mathbb N_0$. However, the same argument shows that one can take $\tilde U_j=\sum_{k=1}^{\infty}\mathcal L^{k}F_{j-k}$ instead of $U_j$ for $j\in \mathbb N_0$. Observe that $\sup_{j\in \mathbb N_0}\|U_j\|<+\infty$ and $\sup_{j\in \mathbb N_0}\|\tilde U_j\|<+\infty$. Moreover $\int_XU_j \, dm=\int_X \tilde U_j \, dm=0$. Furthermore, by \eqref{835eq}, 
\[
\begin{split}
|(\tilde U_n-U_n)\circ T^n| &=\left|\sum_{k=n+1}^{\infty}\mathcal L^{k}F_{n-k}\circ T_0^n\right|\leq\sum_{k=n+1}^\infty\left\|\mathcal L^{k}F_{n-k}\right\|_{L^\infty(m)}\\
&\leq C\sup_{j\in \mathbb Z}\|F_j\|_{\cB_j}\sum_{k=n+1}^\infty e^{-\lambda n}=O(e^{-\lambda n})\to 0,
\end{split}
\]
when $n\to \infty$, and thus~\eqref{LimExp} (with $H_n=\tilde U_n$) holds with $q_n=0$, $n\in \bbN_0$.
\end{remark}

\subsection{Proof of Theorem~\ref{MT2}}
We first introduce sequential counterparts of the functionals $\ell_\om^t$ used in the proof of Theorem~\ref{MT}. More precisely, for $t\in \mathbb R$ and $j\in \mathbb N_0$, we define $\ell_j^t\in \cB_j^*$ by 
\begin{equation}\label{functdef}
\ell_j^t(\varphi):=\int_{X_j}e^{-it H_j}\varphi \, dm_j, \quad \varphi \in \cB_j.
\end{equation}
Using the arguments as in the proof of Theorem~\ref{MT}, one can easily show that $\| \ell_j^t\|_{\cB_j^*}\le 1$ and $\ell_j^t\neq 0$ for every $t\in \mathbb R$ and $j\in \mathbb N_0$.

Before proceeding, we introduce some additional notation. For a sequence $G=(G_j)_{j\in \mathbb N_0}$ of maps $G_j\colon X_j\to \mathbb R$, we set
\begin{equation}\label{eq: def Sj,nG}
S_{j, n}G:=\sum_{k=j}^{j+n-1}G_k\circ T_j^{(k-j)}.
\end{equation}
We note that 
\begin{equation}\label{1158}
    \cL_j^{\theta, (n)}\varphi=\cL_j^{(n)}(e^{\theta S_{j, n}F}\varphi) \quad \text{for $\theta \in \mathbb C$, $j\in \mathbb N_0$, $n\in \mathbb N$ and $\varphi \in \cB_j$,}
\end{equation}
where $F=(F_j)_{j\in \mathbb N_0}$.
\begin{lemma}\label{901Lem}
For all $j\in \mathbb N_0$, $n\in\mathbb N$, $t\in\bbR$ and $\varphi  \in \cB_j$, we have
\[
(\mathcal L_j^{it, (n)})^*\ell_{j+n}^t(\varphi)=\ell_j^t(\varphi).
\]
\end{lemma}

\begin{proof}
We have 
    \[\begin{split}
        \left(\cL_j^{it, (n)}\right)^\ast(\ell^t_{j+n}) (\varphi)&= \ell^t_{j+n}(\cL_j^{it, (n)}\varphi)\\
        &=\int_{X_{j+n}} e^{-itH_{j+n}}\cL_j^{it, (n)}(\varphi) \, dm_{j+n}\\
        &=\int_{X_{j+n}} e^{-itH_{j+n}}\cL_j^{(n)}(e^{i t S_{j,n}F}\varphi) \, dm_{j+n}\\
        &=\int_{X_j} e^{-itH_{j+n}\circ T_j^{(n)}} e^{i t S_{j,n}F}\varphi \, dm_j\\
        &=\int_{X_j} e^{-it H_j}\varphi \, dm_j \\
        &=\ell_j^t(\varphi)
    \end{split}\]
    where we used~\eqref{sdual}, \eqref{sobsdec} and~\eqref{1158}.
\end{proof}
Let $\tilde F=(\tilde F_j)_{j\in \mathbb N_0}$ be a sequence of functions given by 
\[
\tilde F_j=F_j-\int_{X_j}F_j \, d\tilde m_j \quad j \in \mathbb N_0,
\]
where $\tilde m_j=(T_0^{(j)})^*m_0$. Note that $\tilde m_0=m_0$ and that the sequence $\tilde F$ also satisfies~\eqref{sobs1} and~\eqref{sobs2}.

\begin{lemma}\label{variancebounded}
We have $\sup_{n\in \mathbb N}\|S_{0, n}\tilde F\|_{L^2(m_0)}<+\infty$.
\end{lemma}

\begin{proof}
For $j\in \mathbb N_0$ and $\theta \in \mathbb C$, let $\tilde{\cL}_j^\theta \colon \cB_j \to \cB_{j+1}$ be a linear operator defined by~\eqref{ttop},  replacing $F_j$ with $\tilde F_j$. We observe that if $(\lambda_j^\theta, v_j^\theta, \phi_j^\theta)$, $j\in \mathbb N_0$ (where $\theta \in U$ and $U$ is a neighborhood of $0$ in $\mathbb C$) is the triplet given by Theorem~\ref{PTS} that corresponds to the family $(\cL_j^\theta)_{j\in \mathbb N_0, \theta \in \mathbb C}$ then the triplet that corresponds to the family $(\tilde{\cL}_j^\theta)_{j\in \mathbb N_0, \theta \in \mathbb C}$ is given by $(\tilde{\lambda}_j^\theta, v_j^\theta, \phi_j^\theta)$, $j\in \mathbb N_0$, where 
\[
\tilde \lambda_j^\theta=e^{-\theta c_j}\lambda_j^\theta \quad \text{and} \quad c_j:=\int_{X_j}F_j\, d\tilde m_j.
\]
Note that $|\tilde \lambda_j^{it}|=|\lambda_j^{it}|$ for $j\in \mathbb N_0$ and $t\in \mathbb R$.
By~\eqref{newadds}, we find that 
\begin{equation}\label{125}\begin{split}
         &\left|(\mathcal L_0^{it, (n)})^*\ell_{n}^t(\varphi)-\left (\prod_{j=0}^{n-1}\lambda_j^{it}\right)\ell_{n}^t(v_{n}^{it})\phi_0^{it}(\varphi)\right|\\
        & =\left|\ell_{n}^t\left(\mathcal L_0^{it, (n)}\varphi-\left (\prod_{j=0}^{n-1}\lambda_j^{it}\right )\phi_0^{it}(\varphi)v_{n}^{it}\right)\right|\\
        &\leq\left\|\mathcal L_0^{it, (n)}\varphi-\left (\prod_{j=0}^{n-1}\lambda_j^{it}\right )\phi_0^{it}(\varphi)v_{n}^{it}\right\|_{\cB_n}\leq cr^n\|\varphi\|_{\cB_0},
    \end{split}\end{equation}
    for $n\in \mathbb N$, $t\in \mathbb R$ and $\varphi \in \cB_0$. Next, it follows from Lemma~\ref{901Lem}, \eqref{930s} and~\eqref{125} that 
\begin{equation}\label{514eq}
    |\ell_0^t (\varphi)|=|(\mathcal L_0^{it, (n)})^*\ell_{n}^t(\varphi)| \le cr^n \|\varphi\|_{\cB_0}+C^2 \left |\prod_{j=0}^{n-1}\lambda_j^{it}\right |,
\end{equation}
for $\varphi \in \cB_0$, $n\in \mathbb N$ and $t\in [-\delta_0, \delta_0]$ for some $\delta_0>0$. 

Assuming that the conclusion of the lemma is not valid, it follows from~ \cite[Theorem 6.5]{DolgHaf}, ~\cite[Proposition 7.1]{DolgHaf} and~\cite[Corollary 28]{DolgHaf0} that for $t\in \mathbb R\setminus \{0\}$ sufficiently close to $0$, 
\[
\lim_{n\to \infty}\prod_{j=0}^{n-1}\tilde \lambda_j^{it}=\lim_{n\to \infty}\prod_{j=0}^{n-1}\lambda_j^{it}=0.
\]
Fix such $t$. By~\eqref{514eq}, 
\[
\ell_0^t(\varphi)=0, \quad \text{for every $\varphi \in \cB_0$.}
\]
Consequently, $\ell_0^t=0$, which results in a contradiction.
\end{proof}
It follows from Lemma~\ref{variancebounded} and~\cite[Theorem 6.5]{DolgHaf} that for each $j\in \mathbb N_0$,
\begin{equation}\label{543eq}
\tilde F_j=U_{j+1}\circ T_j-U_j+\mathcal M_j,
\end{equation}
where $U_j$ and $\mathcal M_j$ satisfy properties (1) and (2) of the statement of Theorem~\ref{MT2}. Observe that~\eqref{secdec} follows readily from~\eqref{543eq}.
From~\eqref{sobsdec} and~\eqref{secdec} we have 
\begin{equation}\label{cj}
c_j=\int_{X_0}F_j\circ T_0^{(j)}\, dm_0=\hat H_{j+1}\circ T_j-\hat H_j-\mathcal M_j \quad \text{for $j\in \mathbb N_0$,}
\end{equation}
where $\hat H_j:=H_j-U_j$. For $j\in \mathbb N_0$, let  $B_j:=\lim\limits_{n\to \infty}S_{j, n}\mathcal M$, where $\mathcal M=(\mathcal M_j)_{j\in \mathbb N_0}$.  

For $j\in \mathbb N_0$ and $t\in \mathbb R$, let $\bar{\cL}_j^{it}\colon \cB_j \to \cB_{j+1}$ be defined as $\cL_j^{it}$ by replacing $F_j$ with $c_j$. That is, 
\begin{equation}\label{1158a}
\bar{\cL}_j^{it}\varphi=\cL_j(e^{it c_j}\varphi)=e^{it c_j}\cL_j \varphi, \quad \varphi \in \cB_j.    
\end{equation}
Note that 
\[
\bar{\cL}_j^{it, (n)}=e^{it \sum_{k=j}^{j+n-1}c_k}\cL_j^{(n)}.\]
 Similarly to~\eqref{526}, it follows from~\eqref{835eq} that 
\begin{equation}\label{858eq}
\left\|(\bar{\cL}_{j}^{it, (n)})^*\phi-e^{it \sum_{k=j}^{j+n-1}c_k}\phi(v_{j+n})m_j\right \|_{\cB_j^*}\le  Ce^{-\lambda n}\|\phi\|_{\cB_{j+n}^*},
\end{equation}
for $j\in \mathbb N_0$, $n\in \mathbb N$, $t\in \bbR$ and $\phi \in \cB_{j+n}^*$.

Furthermore, let $\bar \ell_j^t\in \cB_j^*$ be defined as $\ell_j^t$ (see~\eqref{functdef}) by replacing $H_j$ with $\hat H_j$. 
\begin{lemma}\label{M}
For $j\in \mathbb N_0$, $t\in \mathbb R$ and $\varphi \in \cB_j$ we have
\[
(\bar{\cL}_{j}^{it, (n)})^* \bar  \ell_{j+n}^t(\varphi)=\bar \ell_j^t(e^{-it S_{j, n}\mathcal M}\varphi)\to \bar \ell_j^t(e^{-itB_j}\varphi) \quad \text{when $n\to \infty$.}
\]
\end{lemma}
\begin{proof}
We have 
    \[\begin{split}
        \left( \bar{\cL}_j^{it, (n)}\right)^\ast(\bar\ell^t_{j+n}) (\varphi)&= \bar \ell^t_{j+n}(\bar{\cL}_j^{it, (n)}\varphi)\\
        &=\int_{X_{j+n}} e^{-it\hat H_{j+n}}\bar{\cL}_j^{it, (n)}(\varphi) \, dm_{j+n}\\
        &=\int_{X_{j+n}} e^{-it\hat H_{j+n}}e^{i t\sum_{k=j}^{j+n-1}c_k}\cL_j^{(n)}(\varphi) \, dm_{j+n}\\
        &=\int_{X_j} e^{-it \hat H_{j+n}\circ T_j^{(n)}} e^{i t\sum_{k=j}^{j+n-1}c_k}\varphi \, dm_j\\
        &=\int_{X_j} e^{-it \hat H_j-itS_{j,n}\mathcal M}\varphi \, dm_j \\
        &=\bar\ell_j^t(e^{-itS_{j,n}\mathcal M}\varphi ),
    \end{split}\]
    where we used~\eqref{sdual}, \eqref{cj} and~\eqref{1158a}.    The convergence of $\bar \ell_j^t(e^{itS_{j,n}\mathcal M}\varphi) $ towards $\bar\ell_j^t(e^{-itB_j}\varphi)$ follows by the dominated convergence theorem.
\end{proof}
Applying~\eqref{858eq} for $\phi=\bar \ell_{j+n}^t$ and using Lemma \ref{M}, we get  
\[
\bar \ell_j^t(e^{-itB_j}\varphi)=b_j^t\int_{X_j}\varphi \, dm_j \quad \text{for $j\in \mathbb N_0$, $t\in \mathbb R$ and $\varphi \in \cB_j$,}
\]
where 
\[
b_j^t:=\lim_{n\to\infty}e^{it\sum_{k=j}^{j+n-1}c_{k}}\bar\ell_{j+n}^t(v_{j+n}).
\]
Let $Q_j:=-B_j-\hat H_j$ for $j\in \mathbb N_0$. Observe that \[\bar\ell_j^t(e^{-itB_j}\varphi)=\int_{X_j}e^{itQ_j}\varphi\, dm_j, \quad \text{for $j\in \mathbb N_0$ and $\varphi \in \cB_j$.}\] Hence,
 \begin{equation}\label{1.1}
 b_j^t=b_j^t\int_{X_j} 1\, dm_j=\int_{X_j} e^{itQ_j}\,dm_j    
 \end{equation}
 and 
\begin{equation}\label{2.1}
b_j^t\int_{X_j} e^{-itQ_j} dm_j=\bar\ell_j^t(e^{it\hat H_j})=1.
\end{equation}
 From~\eqref{1.1} and~\eqref{2.1} we have 
\[
\int_{X_j} e^{itQ_j}\, dm_j=\frac{1}{\int_{X_j} e^{-itQ_j}\, dm_j},
\]
and therefore 
$$
\left|\int_{X_j} e^{itQ_j} dm_j\right|^2=\int_{X_j}e^{itQ_j}\,dm_j\int_{X_j}e^{-itQ_j}\,dm_j=1,
$$
for $j\in \mathbb N_0$. 
Using the same arguments as in the proof Theorem~\ref{MT}, we find that each $Q_j$ is a constant, that is,  for every $j\in \mathbb N_0$, there is $a_j\in \mathbb R$ such that $Q_j=a_j$ ($m_j$-a.e.). We conclude that 
\begin{equation}\label{ABq}
-B_j-\hat H_j=a_j, \quad j\in \mathbb N_0.
\end{equation}
In particular, $H_j\in L^p(X_j, (T_0^{(j)})_*m_0)$ for all finite $p\geq 1$. 
Finally, note that 
\[
\int_{X_0}B_j\circ T_0^{(j)}\,dm_0=\int_{X_0}U_j\circ T_0^{(j)}dm_0=0,
\]
yielding $a_j=-\int_{X_0}H_j\circ T_0^{(j)}\,dm_0$ and that for all $1\leq p<\infty$,
and $\sup_j\|H_j\circ T_0^{(j)}-m_0(H_j\circ T_0^{(j)})\|_{L^p(m_0)}<\infty$.
This completes the proof of Theorem~\ref{MT2}.

\section{Random and sequential two-sided subshfits of finite type}\label{SFT sec}
In this section, we will prove Liv\v sic regularity-type results for random and sequential two-sided subshifts of finite type. These results will be used in the next section to treat small perturbations of hyperbolic maps, but they also have applications, for instance, to inhomogeneous elliptic Markov shifts and for Markov shifts in a random dynamical environment (see \cite[Section 2.3]{LLT}). In what follows, given a compact metric space $(M,\rho)$ and $\beta>0$, we denote by $\|\cdot\|_\beta$ the $\beta$-H\"older norm in the space of $\beta$-H\"older maps from $M$ to $\bbR$.

\subsection{Random SFT}
Let $(\Omega, \mathcal F, \mathbb P)$ be an arbitrary probability space, and let $\sigma \colon \Omega \to \Omega$ be an invertible ergodic measure-preserving transformation on $(\Omega, \mathcal F, \mathbb P)$. Let us take a random variable $d_\om$ that takes values in $\bbN$ such that $\esssup d_\om<\infty$. For $\bbP$-a.e. $\om \in \Omega$ let $A^\om(\cdot,\cdot)$ be a matrix with $0-1$ entries of size $d_\om\times d_{\sig\om}$ that is measurable in $\om$. Set $\cA_\om=\{1,2, \ldots ,d_\om\}$. We assume that there is a constant $M\in\bbN$ such that for $\bbP$ a.e. $\om \in \Om$ the matrix $A^{\sigma^{-M}\om}\cdots A^{\sigma^{-2}\om}\cdot A^{\sigma^{-1}\om}$ has only positive entries. We define 
\[
X_\om=\{(x_k)_{k\geq 0}: \ x_k\in \cA_{\sigma^k \om} \ \text{and} \ A^{\sigma^k\om}(x_{k},x_{k+1})=1 \ \text{for $k\ge 0$}\},
\]
and let $T_\om:X_\om\to X_{\sigma\om}$ be the left shift. Then $X_\om$ can be viewed as a random compact subset of the one-point compactification $X$ of $\bbN^\bbN$.  Let $\cX=\{(\om,x): \om\in\Omega, x\in X_\om\}\subset\Omega\times\bbN^\bbN$. We consider the metric $\rho_\om$ on $X_\om$ given by 
\[
\rho_\om(x,y)=2^{-\inf\{k\geq 0: x_{k}\not=y_{k}\}},\quad x=(x_{k})_{k\ge 0}, y=(y_{k})_{k\ge 0}.
\]

Fix some H\"{o}lder exponent  $\beta>0$ and let $\phi_\om \colon X_\om\to \bbR$ be such that $(\om,x)\to\phi_\om(x)$ is measurable, $x\mapsto \phi_\om(x)$ is $\beta$- H\"older and $\esssup_{\om \in \Om}\|\phi_\om\|_{\beta}<\infty$. Let $m_\om$ be the random Gibbs measure generated by $\phi_\om$ and let $L_\om$ be the corresponding random transfer operator (see \cite{Kifer}). Then $(T_\om)_*m_\om=m_{\sigma\om}$ and all the conditions in Section \ref{RDS} hold with the H\"older norm corresponding to any exponent $0<\alpha\leq \beta$, that is, with $\var_\om(\varphi)$ denoting the H\"older constant of $\varphi$ associated with the exponent $\alpha$.

Let us define the random two-sided shift by 
$$
\tilde X_\om=\{(x_{k})_{k\in \bbZ}: \ x_k\in \cA_{\sigma^k \om} \ \text{and} \  A^{\sigma^k\om}(x_{k},x_{k+1})=1 \ \text{for $k\in \bbZ$}\}
$$
which can be viewed as a random compact subset of the one point compactification of $\bbN^\bbZ$. 
Let $S_\om:\tilde X_\om\to\tilde X_{\sigma\om}$ be the left shift. Then the measure $m_\om$ extends naturally to a probability measure (which will also be denoted by $m_\om$) on $\tilde X_\om$ such that $(S_\om)_*m_\om=m_{\sig\om}$. Let $\pi_\om:\tilde X_\om\to X_\om$ be given by 
$$
\pi_\om((x_{k})_{k\in\bbZ})=(x_k)_{k\geq 0}.
$$
Denote 
$$
S_\om^{(n)}=S_{\sig^{n-1}\om}\circ\ldots\circ S_{\sigma\om}\circ S_\om,
$$
and
$$
\cS(\om,x)=(\sigma\om, S_\om x).
$$
Set $\tilde\cX=\{(\om,x): x\in\tilde X_\om\}\subset\Omega\times\bbN^\bbZ$.
Let the metric $\tilde\rho_\om$ on $\tilde X_\om$ be given by 
$$
\tilde\rho_\om(x,y)=2^{-\inf\{|k|: x_{k}\not=y_{k}\}},\quad x=(x_{k})_{k\in \bbZ}, y=(y_{k})_{k\in \bbZ}.
$$

Let $f\colon \tilde \cX\to\bbR$ be a measurable function such that $f(\om,\cdot)\colon\tilde X_\om\to \bbR$ is H\"older continuous with exponent $\beta$ and $\esssup_{\om \in \Om}\|f(\om,\cdot)\|_{\beta}<+\infty$.
Then, arguing as in \cite[Lemma B.2]{DolgHaf}, there are measurable functions $F\colon\cX\to\bbR$ and $u\colon \tilde\cX\to\bbR$ such that 
 $\esssup_{\om \in \Om}\|F(\om,\cdot)\|_{\beta/2}<\infty$,  $\esssup_{\om \in \Om}\|u(\om,\cdot)\|_{\beta/2}<\infty$ and 
 $$
f(\om,x)=F(\om,\pi_\om(x))+u(\cS(\om,x))-u(\om,x).
$$

\begin{theorem}\label{SFT1}
Let $f$ be as above.
Suppose that for some measurable function $H:\tilde\cX\to\bbR$ we have 
 $$
f=H\circ\cS-H.
 $$
 Then  $\|H(\om, \cdot)\|_{\beta/2}<\infty$ for $\mathbb P$-a.e. $\om \in \Om$. Moreover, $\esssup_{\omega \in \Omega}\var_\om(H(\omega, \cdot))<+\infty$, where $\var_\om$ denotes the H\"older constant corresponding to the exponent $\beta/2$. In fact,
\begin{equation*}
H(\om,\cdot)=u(\om,\cdot)+\int_{\tilde X_\om}(H(\om, \cdot )-u(\om,\cdot))\,dm_\om+\sum_{n=0}^\infty L_{\sigma^{-n}\omega}^{(n)}(\tilde F(\sigma^{-n}\omega, \cdot)),\,\,\mathbb P-\text{a.e.,}
\end{equation*}
where 
\begin{equation*}
\tilde F(\omega, \cdot)=F(\omega, \cdot)-\int_{X_\om} F(\om, \cdot)dm_\om
\end{equation*}
\end{theorem}

\begin{proof}
Let $Q\colon\tilde\cX\to\bbR$ be given by $Q:=H-u$. Then
\begin{equation}\label{1143eq}
F_\om\circ\pi_\om=Q_{\sigma \om}\circ S_\om-Q_\omega
\end{equation}
where $F_\om( \cdot)=F(\om, \cdot)$ and $Q_\om(\cdot)=Q(\om, \cdot)$. We claim that $m_\om$-a.e.,  $Q_\om$ depends only on the coordinates $x_{k}$ for $k\geq 0$, that is, it has the form $Q_\om=q_\om\circ\pi_\om, \mu_\om-$a.e. for some measurable function $q_\om:X_\om\to\bbR$. Once this is proven, all the results stated in the theorem follow from Theorem~\ref{MT}.

To prove the claim, first note that by iterating~\eqref{1143eq} we get that for all $n\in \bbN$,
\begin{equation}\label{Q rep}
Q_\om=Q_{\sigma^n\om}\circ S_\om^{(n)}-\sum_{k=0}^{n-1}F_{\sigma^k\om}\circ\pi_{\sigma^k\om}\circ S_\om^{(k)}.    
\end{equation}
Next, for a point $x^{\om}=(x_{k})_{k\in\bbZ}\in \tilde X_\om$, we write $x^\om_{+}=\pi_\om (x^\om)=(x_{k})_{k\geq 0}$. Then the sum on the right-hand side above depends only on $x^\om_{+}$. 
Let $\varepsilon>0$. Applying Lusin's theorem with the compact space $\tilde X_\om$, we find that there is a closed set $E_\om\subset\tilde X_\om$ such that $m_\om(E_\om)\geq 1-\varepsilon$ and the restriction of $Q_\om$ to $E_\om$ is continuous. Since $E_\om$ is closed and $\tilde X_\om$ is compact, the set $E_\om$ is compact, and so the latter restriction is uniformly continuous.  Let $r_\om>0$ be such that if $x^\om,y^\om\in E_\om$ and $\tilde\rho_\om(x^\om,y^\om)\leq r_\om$ then $|Q_\om(x^\om)-Q_\om(y^\om)|<\varepsilon$. 
For $x^\om\in \tilde X_\om$, let us take $z^\om=z^\om(x_{+}^\om)=(z^\om_{k})_{k<0}$ which depends only on $x^\om_0$ (the zeroth coordinate of $x^\om$) such that the point $\textbf{x}^\om=(z^\om,x^\om_{+})$ belongs to $\tilde X_\om$ (namely $\textbf{x}^\om=(\textbf{x}^{\om}_{k})_{k\in \bbZ}$ with $\textbf{x}^\om_k=x_{k}^\om$ for $k\geq0$ and $\textbf{x}^\om_{k}=z_{k}^\om$ for $k<0$). Let us define $A_\om:\tilde X_\om\to[0,\infty)$ by
$$
A_\om(x^\om):=\inf_{n\geq 1}\left|Q_\om(x^\om)-\left(Q_{\sigma^n\om}\circ S_\om^{(n)}(\textbf{x}^\om)-\sum_{k=0}^{n-1}F_{\sigma^k\om}\circ\pi_{\sigma^k\om}\circ S_\om^{(k)}(x^\om)\right)\right|.
$$
Choose $r_0>0$ sufficiently small so that $\mathbb P(\{\om \in \Om: r_\om\geq r_0\})>0$. Then for $\mathbb P$-a.e. $\om \in \Om$,  there are infinitely many $n \in \bbN$ such that $r_{\sigma^n\om}\geq r_0$. Let us take a typical $\om \in \Om$ and $n\in \bbN$ sufficiently large so that $2^{-n}<r_0$ and $r_{\sigma^n \om}\geq r_0$. Take $x^\om\in (S_\om^{(n)})^{-1}(E_{\sigma^n\om})$. Then since the distance between $S_\om^{(n)} x^\om$ and $S_\om^{(n)} \textbf{x}^\om$ does not exceed $2^{-n}$ (which is smaller than $r_0$) we have
$$
|Q_{\sigma^n\om}(S_\om^{(n)} x^\om)-Q_{\sigma^n\om}(S_\om^{(n)}\textbf{x}^\om)|<\varepsilon,
$$
and therefore 
\[
\begin{split}
A_\om(x^\om) &=\inf_{m\ge 1}|Q_{\sigma^m\om}\circ S_\om^{(m)}(x^\om)-Q_{\sigma^m \om}\circ S_\om^{(m)}(\textbf x^\om)| \\
&\le |Q_{\sigma^n\om}\circ S_\om^{(n)}(x^\om)-Q_{\sigma^n \om}\circ S_\om^{(n)}(\textbf x^\om)|<\varepsilon,
\end{split}
\]
where the first equality uses \eqref{Q rep}. Finally, notice that $m_\om((S_\om^{(n)})^{-1}(E_{\sigma^n\om}))=m_{\sigma^n\om}(E_{\sigma^n\om})\geq 1-\varepsilon$.
We thus conclude
$$
m_\om (\{x^\om: A_\om(x^\om)\geq\varepsilon\})\leq\varepsilon.
$$
 Taking $\varepsilon\to0$ we see that $A_\om(\cdot)=0, m_\om$-a.e.
That completes the proof of the claim and the theorem.
\end{proof}

\begin{remark}
A standard proof of Lusin's theorem is based on approximation by simple functions followed by the application of Egorov's theorem. A standard proof of the latter shows that the set $E_\om$ is formed in an elementary way from sets of the form $E_{\om,n,m}=\{x\in\tilde X_\om: |Q_\om(x)-Q_{\om,n}(x)|<1/m\}, m\in\bbN$ where $Q_{\om,n}$ is a sequence of simple functions that converge to $Q_\om$. Thus, the set $E_\om$ is a random compact measurable subset of $\tilde X_\om$. Consequently, we can choose $r_\om$ so that the map $\om\to r_\om$ is measurable. 
\end{remark}
\subsection{Sequential SFT}

Let us first recall the definition of a sequential one-sided shift.
Let $\mathcal A_j=\{ 1, 2, \ldots ,d_j\}$, $j\in\bbZ$ with $\sup_j d_j<\infty$. Let $A^{(j)}$, $j\in \bbZ$ be matrices of sizes $d_j\times d_{j+1}$ with 0-1 entries. We suppose that there exists $M\in\bbN$ such that for every $j\in \bbZ$ the matrix $A^{(j)}\cdot A^{(j+1)}\cdots A^{(j+M)}$ has positive entries. 
Define 
\begin{equation}
\label{DefOneSided}
X_j=\left\{(x_{j, k})_{k=0}^\infty: \,x_{j, k}\in \cA_{j+k} \ \text{and} \ A^{(j+k)}_{x_{j, k}, x_{j, k+1}}=1 \ \text{for $k\ge 0$}\right\}.
\end{equation}
Let $T_j\colon X_j\to X_{j+1}$ be the left shift.  
Consider a metric $\mathsf{d}_j$ on $X_j$ given by 
$$
\mathsf{d}_j(x,y)=2^{-\inf\{k:\, x_{j,k}\not= y_{j,k}\}}, \quad x=(x_{j, k})_{k=0}^\infty, \ y=(y_{j, k})_{k=0}^\infty. 
$$
Then all the conditions in Section~\ref{SDS} hold (see \cite[Section 4]{DolgHaf}) where $(m_j)_{j\in \mathbb Z}$ is an arbitrary sequential Gibbs measure formed by a sequence of H\"older continuous functions with exponent $\beta>0$.  Moreover, for each $j\in \bbN_0$,  $v_j(\varphi)$ is the H\"{o}lder constant of a H\"{o}lder continuous function $\varphi \colon X_j \to \bbC$ with H\"{o}lder exponent $\alpha$, where $0<\alpha\leq \beta$ is some fixed number.

Define the two-sided shift space by
 \begin{equation}
\label{DefTwoSided}
\tilde X_j=\left\{(x_{j, k})_{k=-\infty}^\infty: \,x_{j, k}\in \cA_{j+k} \ \text{and} \ A^{(j+k)}_{x_{j, k}, x_{j, k+1}}=1\ \text{for $k\in \bbZ$} \right \}.
\end{equation}
Let $S_j:\tilde X_j\to \tilde X_{j+1}$ be the left shift, and let 
$
S_j^{(n)}=S_{j+n-1}\circ\ldots\circ S_{j+1}\circ S_j.
$
Consider the metric $\mathsf{\tilde d}_j$ on $\tilde X_j$ given by 
\[
\mathsf{\tilde d}_j(x,y)=2^{-\inf\{|k|:\, x_{j,k}\not= y_{j,k}\}}, \quad x=(x_{j, k})_{k=-\infty}^\infty, \ y=(y_{j, k})_{k=-\infty}^\infty.
\]
Let $\pi_j:\tilde X_j\to X_j$ be
the natural projection given by 
\[
\pi_j((x_{j, k})_{k=-\infty}^\infty)=(x_{j, k})_{k=0}^\infty.
\]
Then the Gibbs measure $m_j$ can be lifted to $\tilde X_j$ (see the proof of \cite[Proposition B.7]{DolgHaf}).
By~\cite[Lemma B.2]{DolgHaf}, given a sequence of functions $f_j\colon \tilde X_j\to\bbR$, $j\in \bbN_0$  with $\sup_j\|f_j\|_\al<\infty$, there are sequences of functions $F_j:X_j\to\bbR$ and $u_j:\tilde X_j\to\bbR$ such that 
 $ \sup_j\|F_j\|_{\al/2}<\infty$,  $\sup_j\|u_j\|_{\al/2}<\infty$ and 
 \begin{equation}\label{CobRed}
 f_j=F_j\circ\pi_j+u_{j+1}\circ S_j-u_j.    
 \end{equation}

\begin{theorem}\label{SFT Thm2}
Let $f_j\colon\tilde X_j\to\bbR$, $j\in \bbN_0$ be functions such that $ \sup_j\|f_j\|_\al<\infty$
and suppose that for some measurable functions $H_j$, $j\in \bbN_0$ we have 
\[
f_j=H_{j+1}\circ S_j-H_j, \quad j\in \bbN_0.
\]
Suppose that either $H_{n_k}\in L^p(m_{n_k})$ and $\sup_k\|H_{n_k}\|_{L^p(m_{n_k})}<\infty$ for some subsequence $(n_k)_k$ of $\bbN$ and $p>1$
or that $H_j$
satisfy the following regularity condition: there are measurable functions $G_{j,n}$ such that for all $j$, 
\begin{equation}\label{H reg}
 \liminf_{n\to\infty}|H_{j+n}\circ S_j^{(n)}(x)-G_{j,n}(x_{j,0},x_{j,1}, \ldots)|=0, \text{ for }m_j\text{-a.e. }x\in \tilde X_j.   
\end{equation}
Then $H_j\in L^t(m_j)$ for all $j\in \bbN_0$ and $1\leq t<\infty$ and $m_j$-a.s. we have
\begin{equation}\label{H rep}
 H_j=u_j+\int_{X_0}(H_j-u_j)\circ T_0^{(j)}\,dm_0+U_j-\sum_{k=j}^\infty\mathcal M_k\circ T_j^{(k-j)}   
\end{equation}
where $U_n$ and $\mathcal M_n$ satisfy all the properties described in Theorem \ref{MT2} applied to functions $F_j$. Moreover, $\sup_j\left\|H_j-\int_{X_j}H_j\,dm_j\right\|_{L^t}<\infty$.
\end{theorem}

\begin{proof}
 We have 
 $$
f_j=F_j\circ\pi_j+u_{j+1}\circ S_j-u_j=H_{j+1}\circ S_j-H_j.
 $$
 Thus, 
 $$
F_j\circ\pi_j=Q_{j+1}\circ S_j-Q_j
 $$
 where $Q_j:=H_j-u_j$. We claim next that $Q_j$ has the form $Q_j=R_j\circ\pi_j$ for some function $R_j$, namely $Q_j$ depends only on the coordinates with the indexes $k, k\geq0$. Indeed,
we have 
 $$
Q_j=Q_{j+n}\circ S_j^{(n)}-\sum_{k=j}^{j+n-1}F_k\circ\pi_k\circ S_j^{(k-j)}.
 $$
 Now, suppose that $\sup_{k}\|H_{n_k}\|_{L^p(m_{n_k})}<\infty$ for some subsequence $(n_k)_k$ of $\bbN$. Then, since $\sup_s \|u_s\|_{\alpha/2}<\infty$, we have $\sup_{k}\|Q_{n_k}\|_{L^p(m_{n_k})}<\infty$. Let $q$ denote the conjugate exponent of $p$, and take $n=n_k-j$ for $k$ large enough.
Set $\bar Q_{j+n}:=Q_{j+n}-m_{j+n}(Q_{j+n})$. Then
 $$
Q_j=\bar Q_{j+n}\circ S_j^{(n)}+m_{j+n}(Q_{j+n})-\sum_{k=j}^{j+n-1}F_k\circ\pi_k\circ S_j^{(k-j)}=:\bar Q_{j+n}\circ S_j^{(n)}+R_{j,n}.
 $$
Let us take a function $g\colon\tilde X_j\to\bbR$ such that $\|g\|_\beta<\infty$, where $\beta>0$ is the H\"older exponent of the sequence of potentials that generate the Gibbs measures. Then by the exponential decay of correlations for Gibbs measures (see \cite[Theorem 3.3]{Nonlin}) there are $C>0$ and $\delta\in(0,1)$ such that 
$$
|m_{j}(g\cdot \bar Q_{j+n}\circ S_j^{(n)})|\leq C\delta^n\|g\|_\beta \|Q_{j+n}\|_{L^1(m_{j+n})}\to 0\,\text{ as }\,n\to\infty.
 $$
 By approximating a function $g\in L^q(m_j)$ in the $L^q$ norm by a H\"older continuous functions and using that $\bar Q_{j+n}\circ S_j^{(n)}$ are bounded in $L^p(m_j)$ we conclude that
 $$
|m_{j}(g\cdot \bar Q_{j+n}\circ S_j^{(n)})|\to 0\,\text{ as }\,n\to\infty,
 $$
 for all $g\in L^q(m_j)$. Thus, $\bar Q_{j+n}\circ S_j^{(n)}$ converges weakly to $0$ in $L^p(m_j)$. That is 
 $$
Q_j=\lim_{k\to\infty}R_{j,n_k-j}
 $$
 weakly in $L^p(m_j)$. Now, since $R_{j,n}=Q_j-\bar Q_{j+n}\circ S_j^{(n)}$ and $(S_j^{(n)})_*m_j=m_{j+n}$ we see that $\sup_{k}\|R_{j,n_k-j}\|_{L^p(m_j)}<\infty$. Therefore, by the Banach-Saks theorem, there is a subsequence of $n_k$ such that the Cesaro averages of $R_{j,n_k-j}$ along that subsequence converge to $Q_j$ in $L^p(m_j)$. Thus, along a further subsequence, the convergence holds $m_j$-almost everywhere. Since the functions $R_{j,n_k-j}$ depend only on $x_{j,0},x_{j,1}, \ldots$, we conclude that $Q_j$ depends only on these variables.

Next, let us assume \eqref{H reg} holds. Since the functions $u_j$ are uniformly H\"older continuous there exist measurable functions $\bar R_{j, n}(x_{j,0},x_{j,1}, \ldots)$ such that 
$$
\lim_{n\to\infty}|u_{j+n}\circ S_j^{(n)}(x)-\bar R_{j,n}(x_{j,0},x_{j,1}, \ldots)|=0, \text{ for }m_j\text{-a.e. $x\in \tilde X_j$. } 
 $$
 Then, writing $\bar R_{j,n}=\bar R_{j,n}(x_{j,0},x_{j,1}, \ldots)$ and $G_{j,n}=G_{j,n}(x_{j,0},x_{j,1}, \ldots)$ we see that along an appropriate subsequence of $n$'s we have
 $$
Q_j(x)=R_{j,n}+G_{j,n}+o(1)-\sum_{k=j}^{j+n-1}F_k\circ\pi_k\circ S_j^{(k-j
)},
 $$
 where the term $o(1)$ converges to $0$ as $n\to\infty$ (along the appropriate subsequence). Therefore, $Q_j$ depends only on $x_{j,0},x_{j,1}, \ldots$, that is, it is a function on $X_j$. Then 
 $$
F_j=Q_{j+1}\circ T_j-Q_j, \quad j\in \bbN_0.
 $$

 In order to complete the proof of the theorem, applying Theorem \ref{MT2} with the sequence of functions $F_j$ we we see that 
 $$
Q_j=\int_{X_0}Q_j\circ T_0^{(j)}\,dm_0+U_j-\sum_{k=j}^\infty\mathcal M_k\circ T_j^{(k-j)} \quad j\in \bbN_0,
 $$
 where $U_n$ and $\mathcal M_n$ satisfy all the properties described in Theorem \ref{MT2}. Recalling that $Q_j=H_j-u_j$, we conclude that~\eqref{H rep} holds.
\end{proof}
 
\section{Small perturbations of hyperbolic maps}
In this section, we prove Liv\v sic-type regularity results for  small random and sequential perturbations of a given hyperbolic map employing symbolic representations and the results in Section~\ref{SFT sec}.

\subsection{Hyperbolic sets.}\label{Subsec1}
Let $M$ be a compact $C^2$ Riemannian
manifold equipped with its Borel $\sig$-algebra $\mathcal G$. Denote by $\mathsf{d}(\cdot,\cdot)$ the induced metric.
Let $T:M\to M$ be a $C^2$ diffeomorphism.

\begin{definition}
A compact $T$-invariant subset $\Lambda\subset M$ is called a {\em hyperbolic set}  for $T$   if  there exists an open set $V$ with compact closure, constants $\la\in(0,1)$ and $\al_0,A_0,B_0>0$  and subbundles $\Gamma^s$ and $\Gamma^u$ of the tangent bundle $T\Lambda$ such that:
\vskip0.1cm
(i) The set $\{x\in M: \text{dist}(x,\Lambda)<\al_0\}$ is contained in a open subset $U$ of $V$ such that $TU\subset V$ and $T|_{U}$ is a diffeomorphism with \[ \sup_{x\in U}\max(\|D_xT\|, \|D_xT^{-1}\|)\leq A_0;\]
\vskip0.1cm
(ii) $T\Lambda=\Gamma^s\oplus\Gamma^u$, $DT(\Gamma^s)=\Gamma^s$, $DT(\Gamma^u)=\Gamma^u$ and the minimal angle between $\Gamma^s$ and $\Gamma^u$ is bounded below by $\al_0$;
\vskip0.1cm
(iii) For all $x\in \Lambda$ and $n\in\bbN$ we have 
$$
\|D_xT^n v\|\leq B_0\la^n\|v\|\,\,\,\forall \,v\in \Gamma_x^s\quad\text{and}
\quad \|D_xT^{-n} v\|\leq B_0\la^{n}\|v\|\,\,\,\, \forall \, v\in \Gamma_x^u.
$$
\end{definition}

\begin{definition}

A hyperbolic set is said to be:
(i) {\em locally maximal} if the set $U$ above can be chosen
so that $\Lambda=\bigcap_{n\in \mathbb{Z}} T^n U$ (that is, $\Lambda$ is the largest hyperbolic set contained in $U$);

(ii) hyperbolic attractor, if in addition, $U$ could be chosen so that $TU\subset U$
 (in the case where $M=\Lambda$, $T$ is said to be {\em Anosov}). 

We say that $\Lambda$ is a {\em basic hyperbolic set} if it is an infinite locally maximal hyperbolic
set such that $T|_\Lambda$ is topologically transitive.
\end{definition}

\vskip-2mm
 Henceforth, we assume that $\Lambda$ is topologically mixing\footnote{
The topological mixing assumption can be made without a loss
of generality. Indeed (see, e.g. \cite[Chapter 8]{Sh}), 
an arbitrary basic set $\Lambda$ can be decomposed
as $\Lambda=\bigcup_{j=1}^p \Lambda_j$ so that
$T\Lambda_j=\Lambda_{j+1\;\;mod\;\;p}$ where $\Lambda_j$ are topologically mixing
basic hyperbolic sets for $T^p$. Then we could apply the results discussed below
to $(T^p, \Lambda_j).$}
basic
hyperbolic set.
\\

A powerful tool for studying hyperbolic maps is given by symbolic representations. 
That is, every topologically mixing basic set $\Lambda$ admits a Markov partition (see \cite[Chapter 10]{Sh})
that gives rise to a 
semiconjugacy $\pi\colon \Sigma\to \Lambda$, where $\Sigma$ is a
topologically mixing subshift of a finite type. Let us denote by $S\colon \Sigma\to\Sigma$ the left shift and by $\cR$ the Markov partition of $\Lambda$ corresponding to the subshift $\Sigma$.

\subsection{Structural stability}\label{Subsec2}
Now, consider a sequence of $C^2$ maps $\cT=(T_j\colon M\to M)_{j\in\bbZ}$. 
 Denote by $\mathsf{d}_1(f,g)$ the $C^1$-distance between $f$ and $g$.
We have the following result (see \cite[Appendix C]{DolgHaf}).

\begin{theorem}
\label{ThStability}
If $\del_{1}(\cT):=\sup_j\mathsf{d}_{1}(T,T_j)$  is sufficiently  small, then there 
is a sequence  of sets $\Lambda_j\subset M$ and homeomorphisms  $h_j\colon \Lambda\to\Lambda_j$ (that we think of as a ``sequential conjugacy")  
such that 
$h_j$ and $h_j^{-1}$ are uniformly H\"older continuous,
  \begin{equation}\label{Cong}
 T_j\Lambda_j=\Lambda_{j+1}\, \text{ and }\, T_j\circ h_j=h_{j+1}\circ T.
 \end{equation}
Moreover $ \sup_j\|h_j-\text{Id}\|_{C^{0}}\to 0$ as $\del_1(\cT)\to 0.$ 
\footnote{Note that in Theorem \ref{ThStability} we can also consider one-sided sequences $(T_j)_{j\geq0}$ since they can be extended to two-sided ones. The reason we consider two-sided sequences  is  because the definition of hyperbolicity requires considering negative times to 
 define the unstable subspaces.}

The sets  $\Lambda_j$, $j\in \bbZ$ are sequentially hyperbolic for the sequence $\cT$ in the  following sense.
They are compact, 
 there exist constants $\la'\in(0,1)$, $\al_1, A_1,B_1>0$  and sequences of subbundles 
$\Gamma_j^s\!\!=\!\!\{\Gamma_{j,x}^s: x\in \Lambda_j\}$ and 
$\Gamma_j^u\!\!=\!\!\{\Gamma_{j,x}^u: x\in \Lambda_j\}$ of the tangent bundle $T\Lambda_j$
 such that, for each $j$:
\vskip0.2cm
(i) for each $j\in \bbZ$, the set $\{x\in M: \mathsf{d}(x,\Lambda_j)<\al_1\}$ is contained in an open  subset $U_j$ of $V$ such that $T_jU_j\subset V$ and $T_j|_{U_j}$ is a diffeomorphism satisfying
$$
\sup_j\sup_{x\in U_j}\max(\|D_xT_j\|, \|D_xT_j^{-1}\|)\leq A_1;
$$
\vskip0.2cm
(ii) $T\Lambda_j=\Gamma_j^s\oplus\Gamma_j^u$, $DT_j(\Gamma_j^s)=\Gamma_{j+1}^s$, $DT_j(\Gamma_j^u)=\Gamma_{j+1}^u$ and the minimal angle between $\Gamma_j^s$ and $\Gamma_j^u$ is bounded below by $\al_1$;
\vskip0.2cm
(iii) For every $n\in\bbN$, $j\in \bbZ$ and $x\in \Lambda_j$, we have 
\begin{equation}\label{AF1}
\|D_xT_j^{(n)} v\|\leq B_1 (\la')^n\|v\| \quad \text{for every}\,\,v\in \Gamma_{j,x}^s,
\end{equation}
and 
\begin{equation}\label{AF2}
\|D_xT_j^{(-n)} v\|\leq B_1 (\la')^{n}\|v\| \quad \text{for every }\,\,v\in \Gamma_{j,x}^u,
\end{equation}
where $T_{j}^{(-n)}=(T_{j-n}^{(n)})^{-1}$;
\vskip0.2cm
(iv) for $j\in \bbZ$, $T_j U_j\subset U_{j+1}$ and $\bigcap_{n=0}^\infty T_{j-n}^{(n)} U_{j-n}=\Lambda_j$.
\end{theorem}

 Let $\pi_j=h_j\circ \pi$, $j\in \bbZ$. Then the family $\pi_j$, $j\in \bbZ$ provides a semiconjugacy between
the sequence $\cT$ and the subshift $\Sigma$ describing the symbolic dynamics
of $T$. More precisely, the maps $\pi_j$ are surjective and
\[
T_j\circ \pi_j=\pi_{j+1}\circ S, \quad j\in \bbZ.
\]

\subsection{Local stable and unstable manifolds}\label{Subsec3}
We fix $j\in \bbZ$.
For $\ve>0$ small enough and $x\in\Lambda_j$ define $W_j^{s}(x,\ve)$ to be the set of all points $y\in \Lambda_j$ such that
$\mathsf{d}(T_j^{(n)}x,T_j^{(n)} y)\leq \ve$ for all $n\in \bbN$ and $\mathsf{d}(T_j^{(n)}x,T_j^{(n)} y)\to 0$ when $n\to \infty$. Similarly, we define $W_j^{u}(x,\ve)$ as the set of all points $y\in \Lambda_j$ such that
$
\mathsf{d}(T_j^{(-n)}x, T_j^{(-n)}y)\leq \ve
$
for all $n\in \bbN$ and $\mathsf{d}(T_j^{(-n)}x,T_j^{(-n)}y)\to 0$ when $n\to \infty$.  Then $W_j^{s}(x,\ve)$ and $W_j^u(x,\ve)$ are manifolds, and the tangent space of $W_j^{s}(x,\ve)$ at $x$ is $\Gamma_{j, x}^s$, while the tangent space of $W_j^{u}(x,\ve)$ at $x$ is $\Gamma_{j, x}^u$ (see  \cite{MR}). Moreover, there are constants $C>0$ and $\del\in(0,1)$ such that for every $j$,
\begin{equation}\label{Stab}
\mathsf{d}(T_{j}^{(n)} x, T_j^{(n)} y)\leq C \del^n \mathsf{d}(x,y)\text{ for all }y\in W_j^s(x,\ve)
\end{equation}
and 
\begin{equation}\label{Uns}
\mathsf{d}(T_j^{(-n)} x, T_j^{(-n)} y)\leq C\del^n\mathsf{d}(x,y)\text{ for all }y\in W_j^u(x,\ve).
\end{equation}
Furthermore (see \cite[Theorem A]{MR}), there exists $r>0$ such that for all $j$ and all $x,y\in\Lambda_j$ with $\mathsf{d}(x,y)\leq r$ the sets $W_j^s(x,\varepsilon)$ and $W_j^u(y,\varepsilon)$ intersect at a single point $z$ denoted by $[x,y]_j$.

Let $(m_j)$ be the sequential Gibbs measures corresponding to some sequence of H\"older continuous functions $\phi_j$ on $\Lambda_j$ with exponent $\beta$ (see \cite[Appendix C]{DolgHaf}). Note that one particular choice is $\phi_j=-\ln\text{Jac}(T_j)$, which results in sequential SRB measures, see \cite[Theorem C.5]{DolgHaf}. 
Applying Theorem~\ref{SFT Thm2}, we have the following result.

\begin{theorem}\label{ThmSDS}
Let $G_j:\Lambda_j\to\bbR, j\in\bbN_0$ be H\"older continuous functions with exponent $\beta$ and $\sup_j\|G_j\|_\beta<\infty$. Suppose that there are measurable functions $H_j\colon \Lambda_j\to\bbR$ such that
$$
G_j=H_{j+1}\circ T_j-H_j \quad  j\in\bbN_0,
$$
where $H_{n_k}\in L^p(m_{n_k})$ and $\sup_k\|H_{n_k}\|_{L^p(m_{n_k})}<\infty$ for some subsequence $(n_k)_k$ of $\bbN$ and $p>1$.

Then $H_j\in L^t(m_j)$ for all $j$ and $1\leq t<\infty$, $\sup_j\left\|H_j-\int_{X_j}H_j\,dm_j\right\|_{L^t}<\infty$, and $m_j$-a.s. we have
\begin{equation}\label{H rep1}
 H_j\circ\pi_j=u_j+\int_{X_0}(H_j-u_j)\circ T_0^{(j)}\,dm_0+U_j-\sum_{k=j}^\infty\mathcal M_k\circ R^{k-j},
\end{equation}
where $R$ is the one-sided subshift corresponding to $S$, $U_n$ and $\mathcal M_n$ satisfy all the properties described in Theorem \ref{MT2} applied with the functions $F_j$ from \eqref{CobRed} with $f_j=G_j\circ\pi_j$ and $u_j$ also comes from \eqref{CobRed}. 

\end{theorem}

\subsection{Liv\v sic type regularity: reduction to two-sided subshifts}
The following result has applications mainly to small random perturbations of $T$ discussed in the next section, but we formulate it in the more general sequential setup.
\begin{theorem}\label{MT3}
Let $F_j\colon\Lambda_j\to\bbR, j\in\bbZ$ be a sequence of H\"older continuous functions with the same exponent $\alpha>0$ and with uniformly bounded H\"older constants. Let $C'>0$ be an upper bound of the H\"older constant of $F_j$. 
Suppose that there are measurable functions $H_j\colon \Lambda_j\to\bbR$, $j\in \bbZ$ such that
\[
F_j=H_{j+1}\circ T_j-H_j \quad \text{for $j\in \bbZ$,}
\]
and that the sequence of functions $H_j\circ\pi_j$, $j\in \bbN_0$ is uniformly  continuous\footnote{meaning that for each $\varepsilon>0$ there is $\delta>0$ such that $|H_j\circ \pi_j(x)-H_j\circ \pi_j(y)|<\varepsilon$ for each $j\in \bbN_0$ and $x, y\in  \Sigma$ whose distance is less than $\delta$.}. Then $H_j$ has a version $\tilde H_j$ that is H\"{o}lder continuous with exponent $\alpha$ and H\"older constants uniformly bounded by $C''C'$, where $C''$ is a constant that depends only on the sequence $\cT$ and not on the functions $F_j$. Moreover,
\[
F_j=\tilde H_{j+1}\circ T_j-\tilde H_j, \quad j\in \bbZ.
\]    
\end{theorem}

\begin{proof}
The proof is a modification of the proof of \cite[Theorem 2]{PP}.
Fix some index $j$ and take two points $\tilde x_j=\pi_jx_j=h_j(\pi x_j)$ and $\tilde y_j=\pi_j y_j=h_j(\pi y_j)$ in $\Lambda_j$ such that $\mathsf{d}(\tilde x_j,\tilde y_j)<r$. Let $\tilde z_j=[\tilde x_j,\tilde y_j]_j$. Write $\tilde z_j=\pi_jz_j=h_j(\pi z_j)$.  Let us assume that both $\pi x_j$ and $\pi y_j$ are two periodic points of $T$ such that the entire (finite) $T$-orbit of both points does not intersect $\partial\cR$. The collection of such pairs of points is dense in $\Lambda\times\Lambda$. Next, using that $h_j$ are uniformly H\"older continuous and the second equality in~\eqref{Cong}, it follows from \eqref{Stab} and \eqref{Uns} that there exist $\underline C>0$ and $\te\in(0,1)$ such that for all $n\in\bbN$ we have
\[
\mathsf{d}(T^n(\pi z_j),T^n(\pi x_j))\leq \underline C\te^n    \]
and
\[
\mathsf{d}(T^{-n}(\pi z_j),T^{-n}(\pi y_j))\leq \underline C\te^n.    
\]
In particular, there exists $N\in\bbN$ such that for all $n\geq N$ the points $T^n(\pi z_j)$ and $T^n(\pi x_j)$ belong to the interior of a same partition element of $\cR$ and the points $T^{-n}(\pi z_j)$ and $T^{-n}(\pi y_j)$ belong to the interior of a same partition element of $\cR$. Thus, the points $\pi x_j,\pi y_j$ and $\pi z_j$ have a unique symbolic representation, and the coordinates indexed by $n\geq N$ of $x_j$ and $z_j$ coincide, while the coordinates indexed by $n\leq-N$ of $y_j$ and $z_j$ coincide.
In particular,
\begin{equation}\label{1}
\lim_{n\to\infty}d(S^nx_j,S^nz_j)=0
\end{equation}
 and 
 \begin{equation}\label{2}
\lim_{n\to\infty}d(S^{-n}z_j,S^{-n}y_j)=0,
\end{equation}
where $d$ denotes the metric on $\Sigma$ given by $d(a,b)=2^{-\inf\{|k|: a_k\not=b_k \}}$,  $a=(a_k)_k$ and $b=(b_k)_k$. 

Next, write
\[
\begin{split}
& -\sum_{k=0}^{n-1}(F_{j+k}(T_j^{(k)}\tilde x_j)-F_{j+k}(T_j^{(k)}\tilde z_j))-\sum_{k=1}^{n}(F_{j-k}(T_{j}^{(-k)}\tilde y_{j})-F_{j-k}(T_j^{(-k)}\tilde z_j))\\
&=H_j(\tilde x_j)-H_j(\tilde y_j)+(H_{j-n}(T_j^{(-n)}\tilde y_j)-H_{j-n}(T_j^{(-n)}\tilde z_j))\\
&\phantom{=}
+(H_{j+n}(T_j^{(n)}\tilde z_j)-H_{j+n}(T_j^{(n)}\tilde x_j)).
\end{split}\]
Since  $H_j\circ\pi_j$, $j\in \bbN_0$ are uniformly continuous,  we conclude that the last two terms on the right-hand side above converge to $0$. Indeed, by \eqref{1},
\[
\begin{split}
|H_{j+n}(T_j^{(n)}\tilde z_j)-H_{j+n}(T_j^{(n)}\tilde x_j)|&=|H_{j+n}\circ\pi_{j+n}(S^n z_j)-H_{j+n}\circ \pi_{j+n}(S^nx_j)|\to 0.
\end{split}\]
 The proof that the other term converges to $0$  proceeds similarly using~\eqref{2}.
We conclude that 
\[
\begin{split}
&H_j(\pi_jx_j)-H_j(\pi_j y_j)=H_j(\tilde x_j)-H_j(\tilde y_j)\\
&=
-\lim_{n\to\infty}\Bigg(\sum_{k=0}^{n-1}\left(F_{j+k}(T_j^{(k)}\tilde x_j)-F_{j+k}(T_j^{(k)}\tilde z_j)\right)\\
&\phantom{=}+\sum_{k=1}^{n}\left(F_{j-k}(T_{j}^{(-k)}\tilde y_{j})-F_{j-k}(T_j^{(-k)}\tilde z_j)\right)\Bigg).
\end{split}
\]
On the other hand, using \eqref{Stab} for $k\geq0$ we have
\[
\left|F_{j+k}(T_j^{(k)}\tilde x_j)-F_{j+k}(T_j^{(k)}\tilde z_j))\right|\leq C'\mathsf{d}(T_j^{(k)}\tilde x_j, T_j^{(k)}\tilde z_j)^\alpha \le C (C')^\alpha\left(\delta^k\mathsf{d}(\tilde x_j, \tilde y_j)\right)^\alpha,
\]
and using \eqref{Uns} we have
\[
\left|F_{j-k}(T_{j}^{(-k)}\tilde y_{j})-F_{j-k}(T_j^{(-k)}\tilde z_j)\right|\leq  C(C')^\alpha \left(\delta^k\mathsf{d}(\tilde x_j, \tilde y_j)\right)^\alpha,
\]
where $\alpha$ is the H\"older exponent of $F_j$'s. 
We thus conclude that there exists $C''>0$ such that 
\[
\left|H_j(\pi_jx_j)-H_j(\pi_j y_j)\right|\leq C''\left(d_j(\tilde x_j,\tilde y_j)\right)^\alpha= C''\left(d_j(\pi_jx_j,\pi_jy_j)\right)^\alpha.
\]
Now, this inequality holds for a dense set of pairs of points $\pi_j x_j$ and $\pi_j y_j$ within the set $\{(\tilde x_j,\tilde y_j)\in\Lambda_j\times\Lambda_j: \mathsf{d}(\tilde x_j,\tilde y_j)\leq r\}$.
Thus, $H_j$ has a version which is H\"older continuous with the same exponent $\alpha$ as $F_j$ and H\"older constant not exceeding $2C''$.
\end{proof}

\begin{remark}
A family $(T_j)_{j\in \bbZ}$ of $C^2$ maps $T_j\colon M\to M$ is said to be an Anosov family (see \cite{AF,MR}) if there are constants $\lambda'\in (0, 1)$, $\alpha_1, A_1, B_1>0$ and a sequence of subbundles $\Gamma_j^s\!\!=\!\!\{\Gamma_{j,x}^s: x\in M\}$ and
$\Gamma_j^u\!\!=\!\!\{\Gamma_{j,x}^u: x\in M\}$ of the tangent bundle $TM$ such that~\eqref{AF1} and~\eqref{AF2} hold for $n\in \bbN$, $j\in \bbZ$ and $x\in M$. We note that we can allow $M$ to depend on $j$, that is, $T_j\colon M_j \to M_{j+1}$ where $M_j$, $j\in \bbZ$ is a sequence of compact $C^2$ Riemannian manifolds. Moreover, it is not necessary to restrict to the case when the hyperbolicity is global, that is, just as in Section~\ref{Subsec2}, the hyperbolicity requirements can be posed on a sequence $(\Lambda_j)_{j\in  \bbZ}$ of compact subsets of $M$ satisfying the first equality in~\eqref{Cong}.

The discussion in Section~\ref{Subsec2} shows that if $T$ is an $C^2$ Anosov diffeomorphism on $M$ and $(T_j)_{j\in \bbZ}$ is a sequence of $C^2$ maps $T_j\colon M\to M$ such that $\sup_j \mathsf{d}_1(T_j, T)$ is sufficiently small, then $(T_j)_{j\in \bbZ}$ is an Anosov family. 
On the other hand, it is possible to construct examples of Anosov families which are not of this type. In fact, examples of Anosov families can be constructed that are not sequences of Anosov diffeomorphisms (see~\cite[Example 3]{AF}).

In what follows, we will explain how to adapt the proof of Theorem \ref{MT3} to Anosov families. We assume that the Markov partition $\cR_n$ at time $n$ satisfies the following regularity condition. For every $\delta>0$ the volume of the set of points in $M$ whose distance from $\partial\cR_n$ does not exceed $\delta$ is smaller than $C\delta$ for some constant $C>0$ that does not depend on $n$. This condition is needed in order to avoid using periodic points that are not present for more general Anosov families, as will be explained below. 
 The proof of Theorem \ref{MT3} proceeds similarly with the following modifications. We fix $\varepsilon>0$ small enough and consider points $\tilde x_j$ and $\tilde y_j$ such that 
 $$
\text{dist}(T_j^{(n)}\tilde x_j,\partial\cR_{j+n})\geq \frac{\varepsilon}{n^2}
 $$
 and 
 $$
\text{dist}(T_j^{(-n)}\tilde y_j,\partial\cR_{j-n})\geq \frac{\varepsilon}{n^2}.
 $$
 Notice that the volume measure on $M\times M$ of such pairs of points is at least $1-C'\varepsilon$ for some $C''>0$.
 Then using \eqref{Stab} and \eqref{Uns} we find that the point $\tilde z_j$ must have a unique coding and \eqref{1} and \eqref{2} hold. Arguing like at the end of the proof of Theorem \ref{MT3} we find that on a set of measure greater or equal to $1-C''\varepsilon$ the function $H_j$ is H\"older continuous with H\"older constant, which does not depend on $\varepsilon$. Now we get the H\"older continuity almost everywhere with respect to the volume measure by taking $\varepsilon\to0$.
\end{remark}

\subsection{Application to small random perturbations of hyperbolic maps}
Let $M, T$ and $\Lambda$ be as in the beginning of Section \ref{Subsec1}.
Let $(\Omega, \mathcal F, \mathbb P)$ be an arbitrary probability space, and let $\sigma \colon \Omega \to \Omega$ be an invertible ergodic measure-preserving transformation on $(\Omega, \mathcal F, \mathbb P)$. Let $T_\om \colon M\to M, \om\in\Omega$ be a family of maps such that $(\om,x)\to T_\om(x)$ is measurable and $x\mapsto T_\om(x)$ is $C^2$ for $\om \in \Om$. We assume that $\mathsf{d}_1(T_\om,T)<\varepsilon$ for $\bbP$ a.e. $\om \in \Om$ for some sufficiently small constant $\varepsilon>0$.  Let $\tau \colon \Omega \times M \to \Omega \times M$ be the associated skew-product transformation as in~\eqref{tau}. Then, there are random sets $\Lambda_\om$, $\om \in \Om$ such that $T_\om:\Lambda_\om\to\Lambda_{\sigma\om}$ and all the properties in Sections \ref{Subsec1}-\ref{Subsec3} hold for $\bbP$-a.e. $\om$ with $T_{\sigma^j\om}$ instead of $T_j$ and $\Lambda_{\sigma^j \om}$ instead of $\Lambda_j$, and with constants that do not depend on $\om$.
Let $m_\om$ be a random Gibbs measure associated with a random function $\phi_\om:\Lambda_\om\to\bbR$ such that $\esssup_{\om\in\Omega}\|\phi_\om\|_\beta<\infty$ (this includes random SRB measures when taking $\phi_\om=-\ln\text{Jac}(T_\om)$). Denote $\Delta=\{(\om,x): \om\in\Omega, x\in\Lambda_\om\}\subset\Omega\times M$.

By combining Theorem \ref{MT3} and Theorem \ref{SFT1}, we obtain the following result.
\begin{theorem}\label{ThmRDS}
Let $G\colon \Delta\to\bbR$ be a measurable function such that for $\bbP$-a.e. $\om\in \Omega$ the function $G(\om,\cdot)$ is H\"{o}lder continuous with exponent $\beta>0$ and $\esssup_{\om \in \Om}\|G(\om, \cdot)\|_\beta<\infty$. Suppose that there is a measurable function $H\colon \Delta\to\bbR$ such that 
$$
G=H\circ\tau-H.
$$
Then for $\bbP$-a.e. $\om\in \Omega$ the function $H(\om,\cdot)$ is H\"older continuous with exponent $\beta$ and with H\"older constants bounded by some constant $C'$. Moreover, $\esssup_{\om\in\Omega}\|H(\om,\cdot)-\int H(\om,x)\,dm_\om(x)\|_{\beta}<\infty$.  
\end{theorem}

\medskip{\bf Acknowledgements.}
L. Backes would like to thank the University of Rijeka for its generous hospitality, where part of this work was developed.
L.~Backes was partially supported by a CNPq-Brazil PQ fellowship under Grants No. 307633/2021-7 and 304806/2024-2.
 This paper has been funded by European Union – NextGenerationEU-Statistical properties of random dynamical systems and other contributions to mathematical analysis and probability theory (Davor Dragi\v cevi\' c).

\end{document}